\documentclass[10pt]{amsart}

\usepackage{amssymb,amsmath}
\usepackage{amsthm}
\usepackage{esint}
\usepackage{xcolor}
\usepackage{comment}
\usepackage{mathrsfs}

\usepackage{graphicx}

\title[Density-constrained Chemotaxis and Hele-Shaw Flow]{Density-constrained Chemotaxis and Hele-Shaw flow}

\author[I. Kim]{Inwon Kim}
\address{Department of Mathematics, UCLA,  Los Angeles, CA} 
\email{ikim@math.ucla.edu}

\author[A.Mellet]{Antoine Mellet}
\address{Department of Mathematics, University of  Maryland, College Park, MD}
\email{mellet@umd.edu}

\author[Y. Wu]{Yijing Wu}
\address{Department of Mathematics, University of  Maryland, College Park, MD}
\email{yijingwu@umd.edu}

\thanks{I. Kim was partially supported by NSF Grant DMS-1900804. \\
A. Mellet was partially supported by NSF Grant DMS-2009236.
}

\def\R{\mathbb R}
\def\eps{\varepsilon}
\def\e{\varepsilon}
\def\P{\mathcal P}

\def\vphi{\varphi}
\def\pa{\partial}
\def\na{\nabla}
\def\div{\mathrm{div}\,}

\def\BV{\mathrm{BV}}

\newcommand{\J}{\mathscr{J}}

\newcommand{\F}{\mathscr{F}}

\def\H{\mathcal H}

\numberwithin{equation}{section}

\newtheorem{theorem}{Theorem}[section]
\newtheorem{theorem*}{Theorem}
\newtheorem{remark}[theorem]{Remark}
\newtheorem{lemma}[theorem]{Lemma}

\newtheorem{proposition}[theorem]{Proposition}
\newtheorem{corollary}[theorem]{Corollary}
\newtheorem{definition}[theorem]{Definition}

\begin{document}

\begin{abstract}
We consider a model of congestion dynamics with chemotaxis, where the density of cells follows the chemical signal it generates, while observing an incompressibility constraint.
We show that when the chemical diffuses slowly and attracts the cells strongly, then
the dynamics of the congested cells is well approximated by a surface-tension driven free boundary problem. 
More precisely, we show that in this limit the density of cell converges to the characteristic function of a set whose  evolution is described by a Hele-Shaw free boundary problem with surface tension.
Our problem is set in a bounded domain, which leads to an interesting analysis on the limiting boundary conditions for the density function.
Namely, we prove that the assumption of Robin boundary conditions for the chemical potential leads to a contact angle condition for the free interface.
\end{abstract}

\maketitle

\section{Introduction}
\subsection{A model for chemotaxis with density constraint}
The classical parabolic-elliptic Patlak-Keller-Segel model for chemotaxis reads:
$$
\begin{cases}
\pa_t \rho - \mu \Delta \rho +\chi \div(\rho \na \phi)=0 ,\\
 \eta  \Delta \phi  +\theta \rho - \sigma\phi =0, \end{cases}
$$
where $\rho$ denotes the cell density and $\phi$ the concentration of some chemical.
The nonnegative parameters $\mu$ and $\eta$ are the cell and chemical diffusivity, $\chi$ is the cell sensitivity, and $\theta$ and $\sigma$ describe the production and degradation of the chemical (see \cite{Keller_Segel}, \cite{Patlak}, \cite{HLP}). 

%\textcolor{blue}{
%For convenience, we replace $\rho$ with $\rho/\rho_M$ and $\phi$ with $\phi/(\rho_M \theta)$ to get:
%\begin{equation}\label{eq:transpcons}
%\begin{cases}
%\pa_t \rho - \mu \Delta \rho+\beta \div(\rho P_{C(\rho)} (\na \phi))=0 , \qquad\rho\leq 1 \\
% \sigma \phi -  \eta\Delta \phi = \rho.
%\end{cases}
%\end{equation}
%with $\beta = \chi \rho_M \theta$.
%}\medskip

In this model, the diffusion  competes with the aggregating potential $\phi$, leading to the well-known phenomena  of concentration  and finite time blow-up of the density (see e.g. \cite{DP}, \cite{HV}).
In order to investigate the behavior of the density $\rho$ after saturation occurs we take into account the incompressibility of the cells by imposing a constraint 
$\rho\leq \rho_M$.  We replace
$\rho$ with $\rho/\rho_M$ and $\phi$ with $\phi/(\rho_M \theta)$ so that $\rho_M=\theta=1$ and denote  $\bar\chi = \chi \rho_M\theta$.  We are then led to the equation (see \cite{S_survey,KMW} for details):
\begin{equation}\label{eq:weakHS}
\begin{cases}
\pa_t \rho -\mu\Delta \rho + \bar\chi\div( \rho \na \phi) -\Delta  p=0 , \qquad \rho\leq 1\\
 \sigma \phi -  \eta\Delta \phi = \rho,
 \end{cases}
\end{equation}
where the pressure $p$ is a Lagrange multiplier for the contraint $\rho\leq1$, and satisfies
$$ p\geq 0, \qquad p(1-\rho) = 0 \mbox{ a.e.}$$
Similar models have been used in particular in the study of congested crowd motion (see \cite{S_survey}).
% where $\na \phi - \na p$ is interpreted as the projection of the desired velocity field $\na \phi$ onto the set  of admissible velocities, which preserve the constraint $\rho\leq1$. \textcolor{olive}{I think it can be a bit confusing here, because the form of equation should be different here. maybe don't explain this perspective unless we write the above equation in terms of $\div (\rho\nabla p)$?} 
The conditions on $p$ can also be expressed by writing $p\in P(\rho)$ with
\begin{equation}\label{eq:P} P(\rho):= \begin{cases}
0 & 0\leq \rho <1 \\
[0,\infty) &  \rho =1
\end{cases}
\end{equation}
which is sometimes referred to as the Hele-Shaw graph.

\medskip

In a companion paper \cite{KMW}, we proved the existence and uniqueness of a weak solution for \eqref{eq:weakHS}
 and investigated its relation to some free boundary problems.
In this paper, we will investigate the singular limit of strong attraction ($\chi\gg\mu $) and small chemical diffusion ($\eta\ll1$) and prove that the model is asymptotically close to a Hele-Shaw free boundary problem with surface tension \eqref{eq:HSSTintro}. This establishes the first rigorous link between a general Chemotaxis system and Hele-Shaw flow with surface tension (to the best of our knowledge).

\medskip

We also aim to analyze the behavior of the solutions of \eqref{eq:weakHS} near a fixed boundary, by setting the problem in a bounded domain $\Omega \subset \R^d$. In particular we are interested in the effect on the dynamics of different absorption rates of the chemical at the boundary. For full generality, we will use Robin boundary conditions for $\phi$ with a fixed parameter for absorption rate. For the density, we impose Neumann boundary conditions which ensure the conservation of cell density.
%$$ [-\mu \na \rho^\eps +\eps^{-1} \rho^\eps \na {\phi^\eps} -\na p^\eps ]\cdot n=0 \qquad \mbox{ on } \pa\Omega.$$ 
\medskip

%Finally, since the goal of this paper is to characterize the effects of the potential $\phi(x,t)$ on the dynamic of $\rho (x,t)$ when the diffusivity of the chemical is small and the attractive potential dominates the  dynamics, 
Above discussions, by setting $\eta=\eps^2$ and $\bar\chi = \eps^{-1}$ for small $\eps>0$, lead to the system \eqref{eq:weak}-\eqref{eq:phi0}:

\begin{equation}\label{eq:weak}
\begin{cases}
\pa_t \rho  -\mu \na \rho + \div(\eps^{-1}\rho \na \phi-\na p)=0 ,\quad  \mbox{ in } \Omega\times(0,\infty), \qquad p\in P(\rho)\\
(-\mu\na \rho +\eps^{-1}\rho \na \phi-\na p)\cdot n = 0,\quad  \mbox{ on }\pa \Omega\times(0,\infty)\\
\rho(x,0) = \rho_{in}(x)\quad  \mbox{ in } \Omega,
\end{cases}
\end{equation}
with $\phi$ solving
\begin{equation}\label{eq:phi0}
\begin{cases}
\sigma \phi -\eps^2\Delta  \phi = \rho  &  \mbox{ in } \Omega\\
 \alpha \phi + \beta \eps \nabla \phi \cdot n = 0 & \mbox{ on } \pa\Omega.
\end{cases}
\end{equation}
%The Neumann boundary condition in \eqref{eq:weak} is natural and guarantees that $\int_\Omega \rho(t)\, dx$ is preserved. 
%We recall (see the introduction) that the pressure $p(x,t)$, as in the equations of incompressible fluid mechanics, is a Lagrange multiplier for the incompressibility constraint $\rho\leq 1$. In particular (see \eqref{eq:P})
% the condition $p\in P(\rho)$ implies $\rho\leq 1$, $p\geq 0$ and $p(1-\rho)=0$ a.e.  in  $\Omega\times(0,\infty)$.
%\medskip
Note that the scaling of the continuity equation can also be obtained by rescaling the time variable so that we observe the evolution of $\rho$ at a  time scale $\bar t\sim \eps^{-1}/\bar\chi$, under the assumption that $\mu =\mathcal O(\eps\bar\chi)$. 
%The scaling of the coefficient $\alpha$ and $\beta$ in the Robin boundary condition is the scaling in which both coefficient play a non trivial role.

\medskip

Throughout the paper, we assume that $\alpha$, $\beta$ and $\sigma$ are constants satisfying
$$ \sigma>0 ,\quad  \alpha\geq 0,\quad  \beta \geq 0,\quad \alpha+\beta>0.$$
The assumption $\sigma>0$ is important. When $\sigma=0$, the function $\bar \phi = \eps^2\phi$ is the usual Newtonian potential (up to the boundary condition on $\pa \Omega$),  which does not localize in the limit $\e\to 0$: see the discussion below \eqref{eq:HSSTintro}.  By contrast, when $\sigma>0$, we have $\phi\sim \frac 1 \sigma \rho$ when $\eps\ll1$ and the effect of $\eps^{-1}\na \phi$ on the dynamic of the saturated regions is akin to that of surface tension.

\medskip

\subsection{Relation to Hele-Shaw free boundary problems}

When $\mu=0$,  \eqref{eq:weak}-\eqref{eq:phi0} is a weak formulation for the free boundary problem
\begin{equation}\label{eq:HS1}
\begin{cases}
\rho\in[0,1), \quad p =0, \qquad \pa_t \rho  + \eps^{-1}  \div(\rho \na \phi) =0  & \mbox{ in } \Omega\setminus \Omega_s(t) \\
\rho =1, \quad p>0, \qquad \Delta p =\eps^{-1} \Delta \phi & \mbox{ in } \Omega_s(t),
\end{cases}
\end{equation}
where $\Omega_s(t) = \{ \rho(t)=1\}$ denotes the saturated density set and the 
free boundary $\Sigma(t) =\pa\Omega_s(t)\cap \Omega$ moves according to the velocity law
\begin{equation}\label{eq:HS2}
 (1-\rho|_{\Omega_s^C}) V = (-\na p+\eps^{-1} \na\phi)\cdot \nu  |_{\Omega_s}.
 \end{equation}
(Here $V$ denotes the outward normal velocity of $\Sigma(t)$ and $\nu$ denotes the outward normal of $\Omega_s(t)$).
In particular, when the density is a characteristic function $\rho (x,t)= \chi_{\Omega_s(t)}(x)$, we recognize the usual one phase Hele-Shaw problem without surface tension, 
which we can write with the modified pressure $q = p +\eps^{-1} \rho \left( \frac{1}{2\sigma} - \phi\right)$, as
\begin{equation}\label{eq:weakHSepsq}
\begin{cases}
\Delta q = 0 \mbox{ in } \Omega_s(t), \qquad q = \eps^{-1} \left( \frac{1}{2\sigma} - \phi\right)\mbox{ on } \Sigma(t)\\
V  = -\na q \cdot \nu  \quad  \mbox{ on } \Sigma(t).
\end{cases}
\end{equation}
In other words, in this fully  saturated regime, the chemotaxis system \eqref{eq:weak}-\eqref{eq:phi0}
 can be seen as a free boundary problem describing the motion of the region occupied by the cell, driven by the chemical concentration $\phi$ and the pressure variable $p$. 
Since we obtained \eqref{eq:weak}-\eqref{eq:phi0} by imposing the constraint $\rho\leq1 $, but without requiring $\rho\in\{0,1\}$, it is not clear that we should actually have  $\rho (x,t)= \chi_{\Omega_s(t)}(x)$ in general. In \cite{KMW}, we proved that if $\rho$ is a characteristic function at $t=0$, then this remains true at positive times for \eqref{eq:weak}-\eqref{eq:phi0} when $\mu=0$ .
On the other hand when $\mu>0$ the density is never a characteristic function.  Indeed  in this case the saturated set interacts with the unsaturated part of the density by a Richards-type problem, as shown in \cite{KMW}.

\medskip

Nevertheless, we will show in this paper that, in the limit $\eps\to 0$, the effect of the attractive potential is strong enough to ensure the convergence of $\rho$ to a characteristic function $\chi_{\Omega_s(t)}(x)$ for all $\mu \geq 0$. We will then show that the asymptotic dynamic of $\Omega_s(t)$ is described by the Hele-Shaw free boundary problem with surface tension
\begin{equation}\label{eq:HSSTintro}
\begin{cases}
\Delta q= 0 \mbox{ in } \Omega_s(t), \qquad q  =  \frac{\kappa }{4\sigma^{3/2}} \mbox{ on } \Sigma(t)\\
V  = -\na q \cdot \nu \quad   \mbox{ on } \Sigma(t).
\end{cases}
\end{equation}
where $\kappa$ denotes the mean curvature of the free boundary $\Sigma(t)$ (taken to be positive when  $\Omega_s(t)$ is convex). 
Formally, we can get \eqref{eq:HSSTintro} from
 \eqref{eq:weakHSepsq} by proving that the quantity $\eps^{-1} \left( \frac{1}{2\sigma} - \phi\right)$ is an approximation of the mean-curvature of $\Omega_s$ when $\eps\ll1$. 

\medskip

Note that, as $\sigma$ tends to zero, the weight on surface tension grows  to infinity in \eqref{eq:HSSTintro}. Thus heuristically we expect that the limit density support will re-adjust itself into a ball at time scale of order $\sigma^{3/2}$ (and instantly when $\sigma=0$). This is consistent with the convergence to radial solutions of \eqref{eq:weakHS} when $\sigma=\mu=0$: see \cite{CKY} and \cite{HLP} for further discussions.

\subsection{The presence of bounded domain} 
Our result for $\mu=0$ bears similarities with \cite{JKM}, where the emergence of surface tension and derivation of a Muskat problem is studied via a variational approximation. 
In that paper, the potential $\phi$ solves $\phi_t - \Delta \phi =\rho$ (in $\R^d$) and instead of the Keller-Segel system, the authors considers a discrete-time approximations constructed via a JKO scheme.
A similar variational analysis is performed in \cite{LO} for the $L^2$-based thresholding scheme. 

\medskip

Note that both \cite{JKM} and \cite{LO} consider the setting of periodic torus or entire $\R^d$ for the interaction energy, in which case $\phi$ can be written as a convolution with the heat kernel. Such a representation of $\phi$, as well as the symmetry of the heat kernel in space variables, played an important role in the analysis of the aforementioned papers, in particular when deriving the weak limit equation. 
The fact that our problem is set in a bounded domain presents an interesting challenge to this analysis. In particular this necessitates a more PDE-oriented proof of Proposition~\ref{prop:firstvar}, replacing corresponding proofs in \cite{LO} and \cite{JKM}. Our result appears to be the first that links a Keller-Segel system with a Hele-Shaw flow with surface tension, regardless of the choice of the domain. This connection was also suggested in the very recent paper \cite{Perthame}, where the incompressible limit of a generalized version of Keller-Segel system is investigated. The model is a variant of Cahn-Hilliard equation, which can be seen as a diffuse-interface approximation of our Hele-Shaw flow with surface tension. 
%\textcolor{red}{do we want to discuss the work of Elbar-Perthame- Poulain somewhere? Not sure if it is quite relevant but...}

%\textcolor{olive}{We should also comment on Antoine's nice rearrangement of the energy in section 3 to deal with non-characteristic density, though I am  not sure where this would go.} 

\medskip

Another novel feature of our analysis, also related to the bounded domain, is the characterization of the free boundary behavior near the fixed boundary $\pa\Omega$. Of particular interest, in the context of the singular limit $\eps \to 0$, is how the Robin boundary conditions imposed on $\phi$ plays a role on the dynamic of $\Omega_s(t)$. 
We will show that \eqref{eq:HSSTintro} must be supplemented by the contact angle condition 
\begin{equation}\label{eq:CAC}
\cos(\theta)=\gamma:=-\min\left( 1,\frac{2\alpha}{\alpha+\sqrt \sigma \beta} \right) \,\,\hbox{ on } \Sigma(t) \cap \pa\Omega,
\end{equation}
where $\theta$ is the angle formed by the free surface $\Sigma(t)$ and the fixed boundary $\pa \Omega$, measured from inside of the set $\Omega_s(t)$,  and the fixed boundary $\pa\Omega$ at the triple junction $\Sigma(t)\cap \pa\Omega$ (see Figure 1). 
In particular, for Neuman condition $\alpha=0$ (zero absorption of chemicals), the contact must be orthogonal, while for Dirichlet condition $\beta=0$ (and whenever the absortion rate $\alpha$ is bigger or equal than $\sqrt \sigma \beta$), the contact must be tangential.

\medskip

\subsection{Notations and definitions}

Throughout the  paper $\Omega$ is a smooth bounded domain in $\R^d$.

We will use the following definition of weak solutions of \eqref{eq:weak}-\eqref{eq:phi0}, as in \cite{KMW}:
\begin{definition}\label{def:weak}
The pair of functions $(\rho,p)$ is a weak solution of \eqref{eq:weak}-\eqref{eq:phi0} if 
$\rho \in L^1(0,\infty;L^\infty(\Omega)) \cap C^{1/2}(0,\infty;H^{-1}(\Omega))$, $p\in L^2(0,\infty;H^1(\Omega))$ with
%$\rho\leq 1$, $p\geq 0$ and $p(1-\rho)=0$ a.e.  in  $\Omega\times(0,\infty)$.
% $p\in P(\rho)$
%, \quad p\in ?, \quad \phi \in ?$$
$$0\leq \rho\leq 1,\quad  p\geq 0, \quad (1-\rho) p=0 \quad\mbox{ a.e. in } \Omega\times  (0,\infty)$$
and the followings hold:
\begin{equation}\label{eq:weak11}
\int_\Omega \rho_{in} (x) \zeta(x,0)\, dx + \int_0^\infty \int_\Omega \rho\, \pa_t\zeta + \rho v \cdot \na \zeta \, dx = 0 
\end{equation}
for any function $\zeta\in C^\infty_c(\overline \Omega\times[0,\infty))$ and for some $v\in L^2(\Omega\times(0,\infty),d\rho)$  satisfying 
\begin{equation}\label{eq:weak12}
\int_0^\infty \int_\Omega  \rho v\cdot \xi - \eps^{-1}\rho\na \phi \cdot \xi - \mu \rho \, \div\xi - p\, \div \xi\, dx \, dt= 0
\end{equation}
for any vector field  $\xi \in C^\infty_c(\overline \Omega\times(0,\infty) ; \R^d)$ such that $\xi \cdot n=0$ on $\pa\Omega$ and with $\phi$ given by \eqref{eq:phi0}.
\end{definition}
%We recall that $\P(\Omega)$ is equipped with the Wasserstein distance $W_2$, so the condition
%$\rho \in C^{1/2}(0,\infty;\P(\Omega))$ means that $W_2(\rho(t),\rho(s)) \leq C|t-s|^{\frac1 2 }$. In view of Lemma \ref{lem:tech}, this implies also that $\rho \in C^{1/2}(0,\infty;H^{-1}(\Omega))$.
\medskip

Equality \eqref{eq:weak11} is the usual weak formulation for the continuity equation $\pa_t \rho + \div (\rho v)=0$ with Neumann boundary conditions and initial condition $\rho_{in}$. Equation \eqref{eq:weak12} is equivalent to the equality $\rho v =  \eps^{-1}\rho \na \phi - \na p$ in $L^{2}(\Omega\times(0,\infty))$.
It is written in this way to make it easy to compare with Definition~\ref{def:weak2} below (see \eqref{eq:weak12b}).
\medskip

In \cite{KMW} we prove the existence and uniqueness of a weak solution in the sense of Definition \ref{def:weak}, using the fact that it is a gradient flow with respect to the Wasserstein metric. Here the free energy is given by  
$$ 
\mu \int_\Omega \rho \log \rho\, dx - \frac 1 {2\eps}  \int_\Omega \rho\, {\phi^\eps} \, dx, \hbox{with the constraint } \rho \leq 1. 
$$
This energy structure of the equation will also play a key role in this paper.
Because this energy does not behave well when $\eps\ll1$, we will work instead with the functional
$$
\F_\eps(\rho) 
%&  =  \mu \int_\Omega \rho \log \rho\, dx + \frac 1 {2\sigma \eps}  \int_\Omega \rho\, dx  - \frac 1 {2\eps}  \int_\Omega \rho\, {\phi^\eps} \, dx \\
 = 
 \begin{cases}
\displaystyle   \mu \int_\Omega \rho \log \rho\, dx + \frac 1 {2\sigma \eps}  \int_\Omega \rho\,(1 -\sigma {\phi^\eps}) \, dx, & \mbox{ if } 0\leq \rho(x) \leq 1 \mbox{ a.e.}\\
\infty & \mbox{ otherwise .}
\end{cases}
$$
Since $\int_\Omega \rho\, dx$ is preserved  by the equation we are only adding a constant to the energy, but this constant is important when $\eps\ll1$ (it was proved in \cite{MW} that $\F_\eps(\rho)$ is bounded uniformly in $\eps$ when $\rho=\chi_E \in \BV(\Omega; \{0,1\})$). 
%In addition, we proved in \cite{KMW} that the solution of  \eqref{eq:weak}-\eqref{eq:phi0}  is unique.
The following result was proved in  \cite{KMW}:
\begin{theorem}[\cite{KMW}]\label{thm:existence} 
For any $\eps>0$ and any initial condition $\rho_{in}$ satisfying
$$0\leq \rho_{in} \leq 1\quad \mbox{ a.e.  in } \Omega,$$
there exists a unique $(\rho^\eps,p^\eps)$ 
%the interpolations $\rho^{\tau_k,\eps}$ defined by \eqref{eq:interrhop}  converge uniformly in $[0,T]$ with respect to $W_2$ to $\rho^\eps$  and $p^{\tau_k,\eps}$ converges weakly in $L^2(0,T;H^1(\Omega))$ to $p^\eps$ where $(\rho^\eps,p^\eps)$ is the {\em unique} 
weak solution of \eqref{eq:weak}-\eqref{eq:phi0} in  the sense of Definition \ref{def:weak}.
Furthermore, $\rho^\eps$ satisfies the energy inequality
\begin{equation}\label{eq:energy}
\F_\eps(\rho^\eps (t)) + \int_0^t \int_\Omega |v^\eps |^2 \rho^\eps \, dx\,dt\leq\F_\eps(\rho_{in}) \quad \forall t>0
\end{equation}
with $v^\eps$ defined as in Definition \ref{def:weak}.
\end{theorem}
%We also established in \cite{KMW} that the  pressure $p^\eps(x,t)$ satisfies the complementarity condition 
%$$ p^\eps ( \Delta p^\eps - \eps^{-1}\Delta {\phi^\eps} ) = 0 \mbox{ in }\mathcal D'((0,T)\times \Omega)$$

%Definition \ref{def:weak} only requires that $0\leq \rho\leq 1$. If we have $\rho(t) = \chi_{E(t)}$ and $(\rho,p)$ is a weak solution of  \eqref{eq:weak}-\eqref{eq:phi0} in the sense of Definition \ref{def:weak}, then $E(t)$ is a weak solution of the Hele-Shaw problem
%\begin{equation}\label{eq:HS1}
%\begin{cases}
%\Delta p = \eps^{-1}\Delta \phi & \mbox{ in } E(t) \\
%p=0 & \mbox{ on } \pa E(t)\cap \Omega\\
%(\eps^{-1}  \na \phi-\na p)\cdot n = 0 & \mbox{ on }\pa \Omega\cap E(t)\\
%V  = (-\na p+\eps^{-1}\na\phi)\cdot \nu &  \mbox{ on } \pa E(t)
%\end{cases}
%\end{equation}
%with $\phi$ given by \eqref{eq:phiG} (see for instance \cite{MPQ}).
%This problem can also be written in a more familiar way using the modified pressure $q=p+\eps^{-1} \rho \left(\frac{1}{2\sigma} -\phi\right)$:
%\begin{equation}\label{eq:HS2}
%\begin{cases}
%\Delta q =0 & \mbox{ in } E(t) \\
%q=\eps^{-1}  \left(\frac{1}{2\sigma} -\phi\right) & \mbox{ on } \pa E(t)\cap \Omega\\
%\na q\cdot n = 0 & \mbox{ on }\pa \Omega\cap E(t)\\
%V  = -\na q\cdot \nu &  \mbox{ on } \pa E(t).
%\end{cases}
%\end{equation}

The goal of this paper is to show that when $\eps\ll1$, the solution of \eqref{eq:weak}-\eqref{eq:phi0} given by Theorem \ref{thm:existence} converges to the solution of the following Hele-Shaw problem with surface tension:
\begin{equation}\label{eq:HSST}
\begin{cases}
\Delta q= 0 & \mbox{ in } \Omega_s(t)\\
q  =  \frac{\kappa}{4\sigma^{3/2}} & \mbox{ on } \Sigma(t)=\pa \Omega_s(t) \cap \Omega\\
\na q\cdot n = 0 & \mbox{ on } \pa\Omega \cap \Omega_s(t)\\
V  = -\na q \cdot \nu &  \mbox{ on } \Sigma(t)
\end{cases}
\end{equation}
together with the contact angle condition \eqref{eq:CAC}. Recall that $n$ and $\nu$ respectively denote the outward normal of $\Omega$ and $\Omega_s(t)$ at their boundary points.
%\begin{equation}\label{eq:CAC}
%\cos \theta =\gamma:=-\min\left( 1,\frac{2\alpha}{\alpha+\sqrt \sigma \beta} \right)
%\end{equation}
%at the triple junction $\Sigma(t) \cap \pa\Omega$. 

\medskip

The definition of a weak solution of \eqref{eq:HSST}-\eqref{eq:CAC} is parallel to the Definition \ref{def:weak}:
\begin{definition}\label{def:weak2}
The pair of functions $(\rho,q)$ is a weak solution of \eqref{eq:HSST}-\eqref{eq:CAC} if 
$$\rho \in  L^\infty(0,\infty; BV(\Omega;\{0,1\}))\cap C^{1/2}(0,\infty;\P(\Omega)),\qquad  q\in L^2(0,\infty;(C^s(\Omega))^*)$$ 
for some $s>0$
%, \quad p\in ?, \quad \phi \in ?$$
%$$0\leq \rho\leq 1,\quad  p\geq 0, \quad (1-\rho) p=0 \quad\mbox{ a.e. in } \Omega\times  (0,T).$$
and the followings hold:
\begin{equation}\label{eq:weak11b}
\int_\Omega \rho_{in} (x) \zeta(x,0)\, dx + \int_0^\infty \int_\Omega \rho\, \pa_t\zeta + \rho v \cdot \na \zeta \, dx = 0
\end{equation}
for any function $\zeta\in C^\infty_c(\overline \Omega \times[0,\infty))$ and for some $v\in (L^2(\Omega\times(0,T),d\rho))^d$  satisfying 
\begin{equation}\label{eq:weak12b}
\int_0^\infty \int_\Omega  \rho v \cdot \xi - q\, \div \xi\, dx \, dt= -\frac1 {4\sigma^{3/2}} \int_\Omega \left[  \div \xi - \nu\otimes \nu :D\xi\right] |\na \rho|
+ \frac \gamma {4\sigma^{3/2}} \int_{\pa\Omega} \left[  \div \xi - n\otimes n :D\xi\right]  \rho\,  d\mathcal H^{n-1}(x)
\end{equation}
for any vector field  $\xi \in C^\infty_c(\overline \Omega\times(0,\infty) ; \R^d)$ such that $\xi \cdot n=0$ on $\pa\Omega$.
\end{definition}

This definition, similar to the one given in \cite{LO} and \cite{JKM}, warrant several comments.
\begin{enumerate}
\item The condition $\rho \in  L^\infty(0,\infty; BV(\Omega;\{0,1\}))$ implies that for a.e. $t>0$ we have $\rho(t) = \chi_{\Omega_s(t)}$ for a set $\Omega_s(t) \subset \Omega$ with finite perimeter.

\item In \eqref{eq:weak12b}, $\nu = \frac{\na \rho}{|\na \rho|}$ stands for the $L^\infty$ density of $\na \rho$ with respect to the total variation $|\na \rho|$ (which exists by Radon-Nikodym's differentiation theorem).
Since $\rho(t) = \chi_{\Omega_s(t)} \in BV$, it is also the measure theoretic normal to the boundary $\Sigma(t)=\pa \Omega_s(t)$.
In particular, the term 
$\left( \nu\otimes \nu :D\xi\right) |\na \rho|$ is of the form
$ f(x,\lambda /|\lambda |)d|\lambda|$ with $f$ continuous and $1$-homogeneous and $\lambda =\na \rho$. The integral in  \eqref{eq:weak12b} thus makes sense (see for example \cite{serrin}).

\item Note that $q$ has very low regularity in this definition. Since $q\sim\kappa$ along $\Sigma(t)$, we cannot expect much more regularity on $q$ without improving the regularity of the free boundary $\Sigma(t)$.

\item As in Definition \ref{def:weak}, \eqref{eq:weak11b} is simply the weak formulation for the continuity equation $\pa_t \rho + \div (\rho v)=0$ with Neumann boundary conditions and initial data $\rho_{in}$.
Since $\rho=\chi_{\Omega_s(t)}$, it encodes the velocity law $V = v\cdot\nu$, the incompressibility condition $\div v=0$ in $\Omega_s(t)$ and the Neumann condition $v\cdot \nu=0$ on $\pa \Omega \cap \Omega_s(t)$.

\item By taking test functions $\xi$ supported in either $\{\rho=0\}$  or $\{\rho=1\}$, we see that
 Equation \eqref{eq:weak12b} implies $\na q=0$ in $\{\rho(t)=0\}$ and $v = -\na q$ in $\Omega_s(t)=\{\rho(t)=1\}$. Subtracting a constant if needed, we can in particular assume that $q=0$ in $\{\rho(t)=0\}$.
 For general test functions  $\xi$, and taking into account the right hand side of  \eqref{eq:weak12b} we further get 
 the surface tension condition $q  =  \frac{\kappa }{4\sigma^{3/2}} $ on $\Sigma(t)$ and the contact angle condition  \eqref{eq:CAC}. This can be seen by using the classical formula (for a smooth interface $\Sigma$):
\begin{equation}\label{eq:mc}
\int_\Sigma \div \xi - \nu\otimes\nu : D\xi   = \int_\Sigma \kappa\, \xi\cdot\nu + \int_\Gamma b\cdot \xi
\end{equation}
where $\nu$ is the normal vector to $\Sigma$,  $\kappa$ denotes the mean curvature of $\Sigma$ and $b$ is the conormal vector along $\Gamma = \pa \Sigma$.
Indeed, formally at least, the right hand side of   \eqref{eq:weak12b} is (using the fact that $\xi\cdot n=0$ on $\pa\Omega$):
\begin{align*}
& - \frac1 {4\sigma^{3/2}} \left[  \int_{\Sigma} \kappa \xi\cdot\nu + \int_\Gamma \vec b \cdot \xi \right] +\frac \gamma {4\sigma^{3/2}} \left[  \int_{\pa\Omega \cap \Omega_s} \kappa \xi\cdot n + \int_\Gamma \vec c \cdot \xi \right]\\
& =
 \frac1 {4\sigma^{3/2}} \left[  - \int_{\Sigma} \kappa \xi\cdot\nu +\int_\Gamma  \gamma   \vec c \cdot \xi -\vec b \cdot \xi\right]
\end{align*}
where $\vec b$ and $\vec c$ are unit conormal vectors along $\Gamma = \pa \Sigma \cap \pa\Omega$: $\vec b$ is tangent to $\Sigma$ while $\vec c$ is tangent to $\pa\Omega$ (see Figure \ref{fig:1}).

 	 	\begin{figure}[]
 	 		\includegraphics[width=.4\textwidth]{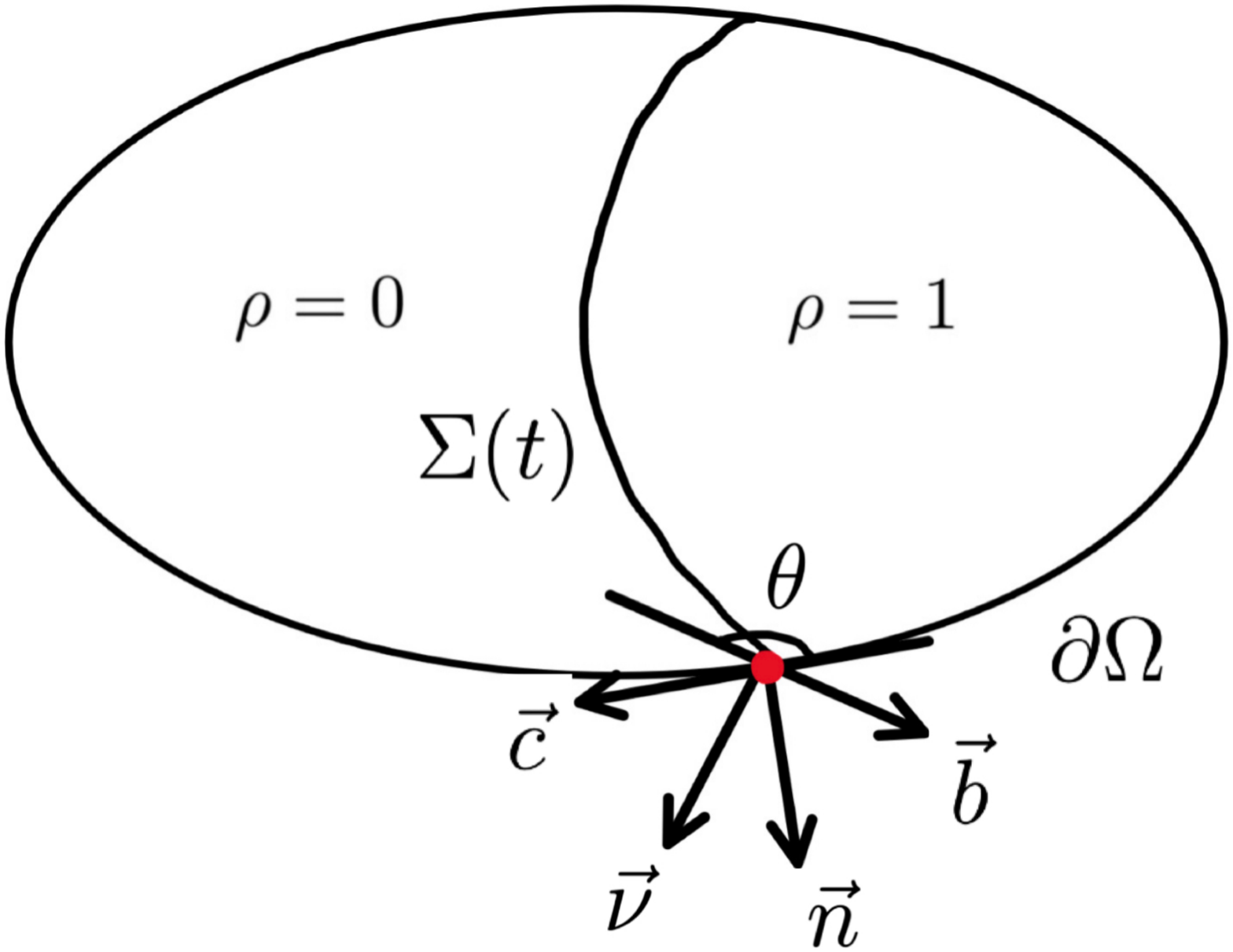}
 	 		\caption{}
 	 		\label{fig:1}
 	 	\end{figure}

%If we integrate by parts  \eqref{eq:weak12b} we thus get $\rho v =\na q$ in $\Omega$ and (recall that $\xi\cdot n=0$ on $\pa \Omega$):
%$$ -\int_{\Sigma} 
% [q] \xi\cdot \nu = 
% \frac1 {2\sigma^{3/2}} \left[  - \int_{\Sigma} \mu \xi\cdot\nu +\int_\Gamma  \gamma   \vec c \cdot \xi -\vec b \cdot \xi\right]
%$$
%where $[q]$ denotes the jump discontinuity of $q$ across $\Sigma$.
Integration by parts in  \eqref{eq:weak12b}  thus reveals that the jump of $q$ across $\Sigma$ must be equal to $\frac{1}{4\sigma^{3/2}}\kappa$. Since $q=0$ in $\{\rho(t)=0\}$, we get $q= \frac{1}{4\sigma^{3/2}}\kappa$ along $\Sigma = \pa E\cap\Omega$. Finally,  the cancellation of the lower dimensional integral requires that the component of the vector $\vec b - \gamma \vec c$ that is tangential to $\pa \Omega$ must vanish.
In particular, we must have $\left[\vec b - \gamma \vec c\right]\cdot \vec c = 0$ and so
$ \vec c \cdot \vec b =  \gamma $,
which gives the contact angle condition 
$$\cos \theta =\gamma.$$

\item  A simple computation show that \eqref{eq:weak12} is equivalent to
$$
\int_0^\infty \int_\Omega  \rho v\cdot \xi  - q\, \div \xi\, dx \, dt= \int_\Omega \eps^{-1}\left(\frac{1}{2\sigma} -\phi\right) \xi \cdot \na \rho
$$
with $q=p+\eps^{-1} \rho \left(\frac{1}{2\sigma} -\phi\right)+\mu\rho$.
Passing to the limit in the right hand side of this equation to derive  \eqref{eq:weak12b}  will be the main result of the second part of the paper (see Proposition \ref{prop:firstvar}) and is at the heart of the relation between the Hele-Shaw model with active potential \eqref{eq:HS2} and the Hele-Shaw model with surface tension \eqref{eq:HSST}.
We will see in particular that the contact angle $\gamma$ depends on the boundary condition for the potential $\phi$ and is given by
$$
\gamma=-\min\left( 1,\frac{2\alpha}{\alpha+\sqrt\sigma \beta} \right).
$$
\end{enumerate}

\medskip

\subsection{Energy}

The proof of the convergence will require some assumptions about the convergence of the energy.
Before stating this assumption, we need to recall a few important facts about the singular part of the energy 
\begin{equation}\label{eq:Jeps}
\J_\eps(\rho) := 
\begin{cases}
 \displaystyle \frac 1 {2\sigma \eps}  \int_\Omega \rho\,(1 -\sigma {\phi^\eps}) \, dx & \mbox{ if $0\leq \rho(x)\leq 1$ a.e. }\\
+\infty &  \mbox{ otherwise}.
\end{cases}
\end{equation}
The properties of $\J_\eps$ when $\eps\ll1$ were studied by two of the authors in  \cite{MW}. The first important result is the following:
%, we recall that the $\Gamma$ convergence of $\J_\eps$ to the perimeter function with an additional term accounting for the two boundary terms in the formula above was the object of the paper.
%More precisely, we proved:
%In that case, we first prove:
\begin{proposition}\cite{MW}\label{prop:4}
Let  $\Omega$ be a bounded open set with $C^{1,\alpha}$ boundary. 
Given a set $E\subset \Omega$ with finite perimeter $P(E,\Omega)<\infty$, we have 
\begin{equation}\label{eq:lims=2}
\lim_{\eps \to 0 } \J_\eps (\chi_E) =  \frac {1}{4 \sigma^{3/2} } \left[ \int_\Omega |\na \chi_E | + 
 \int_{\pa\Omega } \frac{2\alpha}{\alpha+\sqrt \sigma \beta} \chi_E(x) \, d\H^{n-1}(x)\right]  .
\end{equation}
\end{proposition}
%While above result is proved in \cite{MW} with a general class of functions $\alpha$ and $\beta$, here we only consider constant $\alpha$ and $\beta$ for simplicity.
%when
%$\alpha(x)$, $\beta(x)$ are bounded, Lipschitz non-negative functions such that $\alpha(x)+\beta(x)$ is bounded below by a positive constant.
We point out that while the result is proved only for $\sigma=1$ in \cite{MW}, but it can be easily extended to $\sigma\neq 1$ by scaling. More precisely with $\bar \phi = \sigma \phi $, $\bar \eps= \eps/\sqrt\sigma$ and $\bar \beta =\sqrt \sigma \beta$, equation  \eqref{eq:phi0} become
$$
\begin{cases}
\bar \phi -\bar \eps^2\Delta  \bar \phi = \rho  &  \mbox{ in } \Omega\\
 \alpha \bar \phi +\bar \beta \bar \eps \nabla \bar \phi \cdot n = 0 & \mbox{ on } \pa\Omega,
\end{cases}
$$
which is the equation studied in \cite{MW}.

\medskip

Above proposition identifies the limit of $\J_\eps (\chi_E)$.  %For Neumann boundary conditions ($\alpha=0$), we find  $ \frac {1}{4 \sigma^{3/2} } P(E,\Omega)$, but for Dirichlet boundary conditions, we get$ \frac {1}{4 \sigma^{3/2} }\left[  P(E,\Omega)+2 \int_{\pa \Omega}{\chi_E(x)d\H^{n-1}(x)}\right]$. 
However, this functional is not lower-semicontinuous when $\frac{2\alpha}{\alpha+\sqrt \sigma \beta}>1$ and cannot be the $\Gamma$-limit of $\J_\eps$. We can in fact  prove:
\begin{theorem}\label{thm:4}
Let $\Omega$  be a bounded open set with $C^{1,\alpha}$ boundary. 
%Assume further that 
%$\alpha(x)$, $\beta(x)$ are bounded, Lipschitz non-negative functions defined on $\pa\Omega$ such that $\alpha(x)+\beta(x) \geq \sigma>0$ and  that
%$\H^{n-2}\left(\partial \left\{\frac{\alpha}{\alpha+\sqrt\sigma \beta}>\frac 1 2 \right\}\right)<\infty$ (this condition ensures that the function $x\mapsto \frac{\alpha(x)}{\alpha(x)+\sqrt\sigma \beta(x)}$ does not oscillate too much around the value $1/2$). 
%\item
The functional $ \J_\eps $ $\Gamma$-converges, when $\eps\to 0$ to 
$$ \J_{0 }(\rho) 
:=
\begin{cases}
\displaystyle  \frac {1}{4 \sigma^{3/2} } \left[  \int_\Omega |\na \rho|  + \int_{\pa \Omega } \min\left( 1  ,\frac{2 \alpha}{\alpha+\sqrt \sigma \beta} \right) \rho \, d\H^{n-1}(x) \right] & \mbox{ if } \rho \in \BV(\Omega;\{0,1\}) \\
\infty & \mbox{ otherwise.}
\end{cases}
$$

\end{theorem}
This theorem is proved in \cite{MW} when $\J_\eps$ is restricted to characteristic functions, so we show in Appendix \ref{app:energy} how the proof can be generalized to our more general framework.  This extension requires a new formulation of the energy $\J_\eps$, see \eqref{eq:Jepsd2}.

%\medskip
%\textcolor{blue}{this comment below could also go in my opinion}
%For Dirichlet boundary conditions (and more generally when $\alpha\geq\sqrt\sigma \beta$ on $\pa\Omega$), 
%we have $ \J_{0 } (E) =  \frac {1}{4 \sigma^{3/2} }[P(E) +  \int_{\pa \Omega}{\chi_E(x)d\H^{n-1}(x)}]$ while with Neumann boundary conditions ($\alpha=0$) we find $ \J_{0 } (E) = \frac {1}{4 \sigma^{3/2} } P(E,\Omega)$.

\medskip

This $\Gamma$-convergence result suggests that the solution of the gradient flow associated to the energy $\J_\eps$ (which corresponds to equation \eqref{eq:weak}) converges when $\eps\to0$ to a solution of the gradient flow associated to $\J_0$ which is  formally a Hele-Shaw flow \eqref{eq:HSST}-\eqref{eq:CAC}.
This is indeed the result that we want to make precise in the present paper.

\medskip

\subsection{Main results}
We are now able to state the main result:

\begin{theorem}\label{thm:conv1}
Given an initial data $\rho_{in} = \chi_{E_{in}}\in BV(\Omega;\{0,1\})$, $\mu\geq 0$ and a sequence $\eps_n\to 0$, let $(\rho^{\eps_n},p^{\eps_n})$ be the unique solution 
of \eqref{eq:weak}-\eqref{eq:phi0} given by Theorem \ref{thm:existence}.
Then along a subsequence the density $\rho^{\eps_n}(x,t)$  converges strongly in $L^\infty((0,T);L^1(\Omega))$ to 
$$\rho(x,t)\in L^\infty((0,\infty);\BV(\Omega;\{0,1\}))$$ 
and the modified pressure variable $q^{\eps_n}$ (defined by \eqref{eq:modp}) converges to $ q$ weak-$*$ in $L^2((0,T);(C^s (\Omega))^*)$ for any $s>0$. 
Furthermore, $\rho$ satisfies the continuity equation \eqref{eq:weak11b} for some velocity function $v(x,t)$ as well as the 
energy dissipation property
\begin{equation}\label{eq:energy2}
\J_0(\rho(t)) + \int_0^t \int_\Omega |v|^2 \rho \, dx\, dt \leq \J_0(\rho_{in}).
\end{equation}

Finally, if the following energy convergence assumption holds:
\begin{equation}\label{eq:EA}
\lim_{n\to\infty} \int_0^T \J_{\eps_n}( \rho^{\eps_n} (t)) \, dt = \int_0^T \J_0(\rho(t))\, dt
\end{equation}
then the limit $(\rho, q)$ also satisfies the pressure equation \eqref{eq:weak12b} on $(0,T)$. Thus it follows that $(\rho,p)$ 
is a weak solution of \eqref{eq:HSST}-\eqref{eq:CAC} in the sense of Definition \ref{def:weak2}, with initial condition $\rho_{in}$ and contact angle
$$\gamma=-\min\left( 1,\frac{2\alpha}{\alpha+\sqrt \sigma \beta} \right).$$

\end{theorem}

The result also holds if we consider a sequence of initial data $\rho_{in} ^\eps= \chi_{E^\eps_{in}}$ bounded in $BV(\Omega)$, converging strongly to $\rho_{in} = \chi_{E_{in}}$  in $L^1$ and satisfying $\lim  \J_{\eps_n}( \rho^{\eps}_{in})  \to \J_0(\rho_{in})$.
The existence of such a sequence, for any  finite perimeter set $E_{in}$, is proved in \cite{MW} (Proposition 5.3) as part of the $\Gamma $ convergence result.

\medskip

We note that Theorem \ref{thm:conv1} is a conditional result, since it requires the energy convergence assumption \eqref{eq:EA}. 
The analysis  of \cite{MW} implies that we always have
$$
\liminf_{n\to\infty} \int_0^T \J_{\eps_n}( \rho^{\eps_n}(t) ) \, dt\geq \int_0^T \J_0(\rho(t))\, dt,
$$
so \eqref{eq:EA} ensures that there is no  loss of boundary between phases in the limit.
This assumption is rather natural and is similar to the one required for instance in  \cite{LS, LO, JKM}. It is likely that by proceeding as in \cite{CL} one could also obtain a weaker notion of solutions using the theory of varifolds without this assumption, but we do not pursue this direction here.

%Finally, observe that the convergence of the density is slightly better in this case than in Theorem \ref{thm:conv}.
%This is due to the fact that we recover estimates on the  total variation  of $\rho$ via the energy.
%The behavior of the energy when $\eps\ll1$ makes it also quite easy to show that the limiting density is a characteristic function in this case.

\subsection{Outline of the paper}
We begin with deriving two alternative formulas for the energy $\J_\eps$ in the next section, which play a crucial role in our analysis. Section \ref{sec:apiori} collects the main a priori estimates for the $\e$-solutions.
The proof of Theorem~\ref{thm:conv1} is then split between sections 4 and 5.
In Section \ref{sec:conv}, we prove Proposition \ref{prop:rhostrong} which gives the first part of the theorem, namely the strong convergence in $L^1$ of the density toward a characteristic function which satisfies the continuity equation \eqref{eq:weak11b} and the energy inequality  \eqref{eq:energy2}.
 Section \ref{sec:lim} completes the proof of Theorem \ref{thm:conv1} by deriving 
equation \eqref{eq:weak12b} under condition \eqref{eq:EA}. The main step is Proposition \ref{prop:firstvar} which shows that the convergence of the energy \eqref{eq:EA} implies the convergence of the first variation. In the last section, we briefly recall the construction of the JKO scheme used in \cite{KMW} to prove the existence of weak solutions to \eqref{eq:weak}-\eqref{eq:phi0} (Theorem \ref{thm:existence}) and  we state a convergence result similar to Theorem \ref{thm:conv1} for a discrete-time approximation: such a result is of independent interest for numerical applications.

\medskip

\section{Alternate formulas for $\J_\eps$}\label{sec:J}
A crucial tool in our analysis will be a couple of alternate formula for the energy $\J_\eps$.
We recall that the total energy of the model, $\F_\eps$, is given by
$$ \F_\eps(\rho) = \mu\int_\Omega \rho \log\rho\, dx + \J_\eps(\rho),$$ 
where $\J_\eps$ is defined by \eqref{eq:Jeps} and  plays key role in the analysis 
when $\eps\ll1$.
For $\rho$ satisfying the constraint $0\leq \rho\leq1$, we have
$$
\J_\eps (\rho) = 
 \frac 1 {2\sigma \eps}  \int_\Omega \rho\,(1 -\sigma \phi) \, dx 
 = \frac 1 {2\sigma \eps}  \int_\Omega \rho(x) - \sigma \int_\Omega G_\eps(x,y) \rho(y) \, dy  \, dx $$
 for some kernel $G_\eps$. 
A similar energy functional is used in \cite{LO,JKM} with $G_\eps$ is the heat kernel in $\R^d$.
However, we will rely on some different formulations for $J_\eps$ which  make use of the particular equation solved by the function $\phi^\eps$ in our model:
First, we write 
\begin{align*}
\J_\eps (\rho) 
% & = \frac{1}{2\sigma \eps} \int_\Omega \rho - \sigma \phi \rho \, dx \\
 & = \frac{1}{2\sigma \eps} \int_\Omega (\rho-\rho^2) +(\rho^2 - 2\sigma \phi \rho + (\sigma\phi)^2) -  (\sigma\phi)^2  + \sigma \phi \rho \, dx \\
 & = \frac{1}{2\sigma \eps} \int_\Omega \rho(1-\rho)- ( \rho - \sigma \phi )^2 -  \sigma\phi ( \sigma \phi - \rho)  \, dx 
\end{align*}
and using  equation \eqref{eq:phi0} for $\phi$ implies
\begin{align*}
 \J_\eps (\rho)  & = \frac{1}{2\sigma \eps} \int_\Omega\rho(1-\rho)+ ( \rho - \sigma \phi )^2 -  \sigma\phi ( \eps^2 \Delta \phi)  \, dx \nonumber \\
 & = \frac{1}{2\sigma \eps}  \int _{\Omega}\rho(1-\rho) \, dx +  \frac{1}{2\sigma \eps}  \int _{\Omega}(\rho-\sigma \phi)^2 \, dx  + 
\frac{ \eps}{2}  \int_\Omega |\na \phi| ^2\, dx + \frac{1}{2}  \int_{\pa \Omega}\frac{\alpha}{\beta} |\phi|^2 \, d\H^{n-1}(x) 
 \end{align*} 
 when $\beta\neq 0$ and 
 \begin{align*}
 \J_\eps (\rho)  = \frac{1}{2\sigma \eps}  \int _{\Omega}\rho(1-\rho) \, dx +   \frac{1}{2\sigma \eps}  \int _{\Omega}(\rho-\sigma \phi)^2 \, dx  + 
\frac{ \eps}{2}  \int_\Omega |\na \phi| ^2\, dx    
 \end{align*}
when $\beta=0$.
Alternatively, we can write the more symmetric formula (for any $\alpha,\; \beta\geq 0$ with $\alpha+\beta>0$):
\begin{align}
 \J_\eps (\rho)  & =  \frac{1}{2\sigma \eps}\int_\Omega  \rho(1-\rho)\, dx
+  \frac{1}{2\sigma \eps}  \int _{\Omega}(\rho-\sigma \phi)^2 \, dx  + 
\frac{ \eps}{2}  \int_\Omega |\na \phi| ^2\, dx\nonumber \\
 & \qquad + \frac{\eps^2}{2}  \int_{\pa \Omega}\frac{\beta}{\alpha+\beta} |\na \phi\cdot n|^2 \, d\H^{n-1}(x) +
\frac{1}{2}  \int_{\pa \Omega}\frac{\alpha}{\alpha+\beta} |\phi|^2 \, d\H^{n-1}(x) .\label{eq:Jepsd}
 \end{align}
 This formula played a key role in the proof of Proposition \ref{prop:4} and Theorem \ref{thm:4} in \cite{MW}.
 Thanks to the constraint $0\leq \rho\leq 1$, all the terms in this formula are non-negative (without the constraint, the first term will favor values of $\rho$ larger than $1$).
Furthermore, in the regime $\eps\ll1$, the first term will be bounded only for characteristic functions. This observation will be crucial when proving that the limiting density is a characteristic function (even though $\rho^\eps$ may not be).
We also note  that the scaling of the following two terms is consistent with the scaling of the classical Modica-Mortola regularization of the perimeter functional.

 \medskip

Our analysis will also require a slight variation of this formula:
 we can write
$$ 
 \rho(1-\rho)+ (\rho-\sigma \phi)^2 = \rho +\sigma^2\phi^2 - 2\rho\sigma \phi
  = (1-\rho) (\sigma\phi)^2 + \rho(1-\sigma \phi)^2$$
leading to the formula
  \begin{align}
 \J_\eps (\rho)  & =  \frac{1}{2\sigma \eps}\int_\Omega   (1-\rho) (\sigma\phi)^2 + \rho(1-\sigma \phi)^2\, dx
+\frac{ \eps}{2}  \int_\Omega |\na \phi| ^2\, dx\nonumber \\
 & \qquad + \frac{\eps^2}{2}  \int_{\pa \Omega}\frac{\beta}{\alpha+\beta} |\na \phi\cdot n|^2 \, d\H^{n-1}(x) +
\frac{1}{2}  \int_{\pa \Omega}\frac{\alpha}{\alpha+\beta} |\phi|^2 \, d\H^{n-1}(x) .\label{eq:Jepsd2}
 \end{align}
Here also we note that all the terms are non-negative when $0\leq \rho\leq1$. This formula will in particular be crucial in proving the strong convergence of $\rho^\eps$ and when deriving the pressure equation \eqref{eq:weak12b} (see Section \ref{sec:lim}).

\medskip

\section{A priori estimates}\label{sec:apiori}
We now derive the a priori estimates that will be used to prove the convergence of $\rho^\eps$.
We have:
%The following a priori estimates for $\rho^\eps$ follow from the  energy inequality \eqref{eq:energy}:
\begin{lemma}\label{lem:unifestimates}
Let $\rho_{in} \in \BV(\Omega;\{0,1\})$
and  $(\rho^{\eps},p^{\eps})$ be the unique solution 
of \eqref{eq:weak}-\eqref{eq:phi0} given by Theorem~\ref{thm:existence}.  
There exists a constant $C$ depending only on $\int_\Omega |\na \rho_{in}|$ (and in particular independent of $\eps$) such that for all $\eps>0$ we have:
%Let $(\rho^{\eps,\tau},v^{\eps,\tau},E^{\eps,\tau})$ and $(\widetilde{\rho}^{\eps,\tau},\widetilde{v}^{\eps,\tau},\widetilde{E}^{\eps,\tau})$ be the previously constructed piecewise constant and continuous in time interpolations, respectively. Then there exists $C > 0$ independent of $\tau>0$ and $\eps$ such
%that
\item[(i)]  $\J_{\eps}(\rho^\eps(t)) \leq C$ for all $t>0$.
\item[(ii)]  $\int_0^\infty\int_\Omega |v^{\eps} |^2 \rho^\eps \, dx\, dt \leq C$ and  $\| E^{\eps}\|_{L^2(\Omega\times(0,\infty))}\leq C $ %$E^{\eps,\tau} $ and $\widetilde{E}^{\eps,\tau} $ are uniformly bounded in $\mathcal{M}^d([0,T]\times \Omega)$ 
\item[(iii)]  $\| \rho^{\eps}(t)-\rho^{\eps}(s)\| _{H^{-1}(\Omega)} \leq C\sqrt{t-s}$
 for any $0\leq s\leq t$.
%\item[(iii)] $W_2(\widetilde{\rho}^{\eps,\tau}(t),\widetilde{\rho}^{\eps,\tau}(s))\leq C\sqrt{t-s}$ for any $0\leq s\leq t\leq T$ and $\| \widetilde v^{\eps,\tau} \|_{L^2_{\widetilde\rho^{\eps,\tau}}}\leq C$
\end{lemma}

\begin{proof}
We recall that
$$ \F_\eps(\rho) = \mu\int_\Omega \rho \log\rho\, dx + \J_\eps(\rho),$$ 
where (since $0\leq \rho\leq 1$)  $-C\leq \rho\log\rho \leq 0$.
The energy inequality \eqref{eq:energy}
thus  implies
$$ \F_\eps(\rho^{\eps}(t)) \leq \F_\eps(\rho_{in}) \leq \J_\eps(\rho_{in}) \quad \forall t>0.$$
Using Proposition \ref{prop:4}, we see that when
$\rho_{in} = \chi_{E_{in}} \in BV(\Omega;\{0,1\})$, we have  $  \J_\eps(\rho_{in}) \leq C $ for some constant $C$ independent on $\eps$. 
We deduce 
\begin{equation}\label{eq:Jbd}
\J_\eps(\rho^{\eps}(t)) \leq \J_\eps(\rho_{in})-   \mu\int_\Omega \rho \log\rho\, dx \leq C.
\end{equation}

\medskip

The energy  inequality  also gives 
$$\int_0^\infty \int_\Omega |v^\eps |^2 \rho^\eps\, dx\,dt \leq  \F_\eps(\rho_{in}) \leq \J_\eps(\rho_{in})$$
and since $\rho^\eps\leq 1$ (ii) follows immediately.

\medskip

Finally, for a given test function $\psi \in H^1(\Omega)$, the continuity equation \eqref{eq:weak11} implies
$$
\int_\Omega \rho^\eps(x,t) \psi(x)\, dx - \int_\Omega \rho^\eps(x,s) \psi(x)\, dx = \int_s^t \int_\Omega \rho^\eps v^\eps \cdot\na \psi\, dx\, d\tau
$$ 
and so (since $\rho^\eps\leq 1$):
\begin{align*}
\left| 
\int_\Omega \big(\rho^\eps(x,t)- \rho^\eps (x,s) \big)\psi(x)\, dx \right|
& \leq  \left( \int_s^t \int_\Omega |v^\eps |^2 d\rho^\eps \right)^{1/2}\left(\int_s^t \int_\Omega \rho^\eps |\na \psi|^2 \, dx\, d\tau \right)^{1/2}\\
& \leq   \| \psi\|_{H^1(\Omega)}\left( \int_0^t \int_\Omega |v^\eps |^2 d\rho^\eps \right)^{1/2}\left(t-s\right)^{1/2}
\end{align*}
and (iii) now follows from (ii).
\end{proof}

We also need some estimates on $\phi^\eps$, 
solution of \eqref{eq:phi0}. 
%We recall that this solution can be expressed as
%\begin{equation}\label{eq:phiG}
% \phi (x)=\int_\Omega G_\eps(x,y) \rho(y)\, dy
% \end{equation}
%for some Green kernel $G_\eps(x,y): \Omega\times\Omega\to\R^d$.
%Importantly, the Green kernel is not of the form $G(|x-y|)$. But since the equation \eqref{eq:phi0}
%is self-adjoint, we have
%$$
%G_\eps(x,y)=G_\eps(y,x), \qquad \na_x G_\eps(x,y)=-\na_y G_\eps(y,x)
%$$
The maximum principle applied to \eqref{eq:phi0} immediately gives 
\begin{equation}\label{eq:phimax} 
0\leq  \phi(x)\leq 1/\sigma \mbox{ in } \Omega
\end{equation}
and multiplying \eqref{eq:phi0} by $\phi$ and integrating leads to the estimate
\begin{equation}\label{eq:phiH1}
\sigma \| \phi \|_{L^2(\Omega)}^2+ \eps^2 \|\na \phi\|_{L^2(\Omega)}^2 \leq 1/\sigma.
\end{equation}
\medskip

\section{Strong Convergence of $\rho^{\eps}$ and continuity equation.}\label{sec:conv}

The main result of this section is the following proposition which proves the first part of Theorem \ref{thm:conv1}:
\begin{proposition}\label{prop:rhostrong}
Let  $\rho_{in}(x) = \chi_{E_{in}} \in BV(\Omega;\{0,1\})$ and $\rho^\eps(x,t)$ the unique solution 
of \eqref{eq:weak}-\eqref{eq:phi0} given by Theorem \ref{thm:existence}. 
Consider a sequence such that $\eps_n\to 0$. The followings hold:
\item[(i)] There exists a subsequence (still denoted $\eps_n$) along which $\rho^{\eps_n}(t)$ converges uniformly with respect to $t$ in $H^{-1}(\Omega)$  to $\rho(t)$ and $E^{\eps_n}$ converges weakly in $L^2(\Omega\times(0,\infty))$ to $E$.
\item[(ii)] 
There exists $v\in (L^2(\Omega\times(0,\infty),d\rho))^d$ such that $E=\rho v$ and the continuity equation \eqref{eq:weak11b} holds.
\item[(iii)] Up to another subsequence, $\rho^{\eps_n}(t)$ converges to $\rho(t)$ strongly in $L^1(\Omega)$, uniformly in $t$. Furthermore, for all $t>0$ we have 
$$\rho(t)\in BV(\Omega;\{0,1\})$$
(that is $\rho(t)$ is the characteristic function of a set of finite perimeter)
and  the energy inequality \eqref{eq:energy2} holds.
\end{proposition}
Note that (i) and (ii) are classical. The most important statement is thus (iii).

%$\rho^{\eps_n,\tau_n}$ and $\widetilde{\rho}^{\eps_n,\tau_n}$ converge uniformly with respect to $W_2$  to the same limit
%and  $(E^{\eps_n,\tau_n})$ and $(\widetilde{E}^{\eps_n,\tau_n})$ converge narrowly (as bounded measure) to the same limit.

\begin{proof}
The a priori estimates of Lemma \ref{lem:unifestimates} implies (i).
Furthermore, we can pass to the limit 
 in \eqref{eq:weak11} to get
 $$
\int_\Omega \rho_{in} (x) \zeta(x,0)\, dx + \int_0^\infty \int_\Omega \rho\, \pa_t\zeta +E \cdot \na \zeta \, dx = 0 
$$
for any function $\zeta\in C^\infty_c([0,\infty)\times\overline \Omega)$. This is the continuity equation
\begin{equation}\label{eq:cont22}
\begin{cases}
 \pa_t \rho +\div E=0,\\ 
 \rho(x,0)=\rho_{in}(x).
 \end{cases}
 \end{equation}
To complete the proof of (ii) and derive \eqref{eq:weak11b}, we just need to show that $E$ can be written in the form $\rho v$.
This is a classical argument which we recall here:
First recall that the function
$$ \Theta: (\mu,F)\mapsto 
\begin{cases}
\displaystyle \int_0^T\int_\Omega \frac {|F|^2}{\mu}  & \mbox{  if } F \ll \mu \mbox{ a.e. } t\in [0,T]\\
+\infty & \mbox{ otherwise}
\end{cases}
$$
is lower semi-continuous for the weak convergence of measure.
Together with the uniform bound $\Theta (\rho^{\eps_n} ,E^{\eps_n} ) = \int_0^T\int_\Omega \rho ^{\eps_n} |v^{\eps_n}|^2 \leq C$ (see Lemma \ref{lem:unifestimates} (ii)), it implies that $E$ is absolutely continuous with respect to $\rho$ and that there exists $v(t,\cdot) \in L^2 (d\rho(t))$ such that $E = \rho v$.
Inserting this in \eqref{eq:cont22} yields \eqref{eq:weak11b}.

\medskip

The rest of the proof is devoted to (iii).
The fact that we can get stronger convergence for the density is non trivial and is due to the fact that the energy $\J_\eps$ controls the $\BV$ norm of $\phi^{\eps}$, which is close to $\rho^{\eps}$ when $\eps\ll1$. 
To see this, we introduce  the function
$$ F(t) = \int_0^t 2 \min(\tau,1-\tau)\, d\tau = \begin{cases} t^2 & \mbox{ for } 0\leq t\leq 1/2 \\ 2t-t^2 -\frac 1 2 & \mbox{ for } 1/2 \leq t\leq 1.\end{cases} $$
We then have  $F'(\sigma \phi) = 2 \min(\sigma \phi,1-\sigma \phi)$ and so
\begin{align*}
\frac{1}{\sigma^{3/2}}|\na F(\sigma \phi)| 
& \leq 2\frac{1}{\sqrt \sigma}  |\na \phi|\min(\sigma \phi,1-\sigma \phi) \\
& \leq \frac{1}{\sigma \eps} \min((\sigma \phi)^2,(1-\sigma \phi)^2) +  \eps |\na \phi|^2\\
& \leq \frac{1}{\sigma \eps}\left[ (1-\rho)  (\sigma \phi)^2+ \rho (1-\sigma \phi)^2 \right]+  \eps |\na \phi|^2
\end{align*}
(as long as  $0\leq \rho\leq 1$).
This inequality, together with the formula 
\eqref{eq:Jepsd2} for $\J_\eps$
 implies
 %Since $F'(\sigma \phi) = 2 \min(\sigma \phi,1-\sigma \phi)\leq 2 |\sigma \phi-\rho|$ when $\rho$ is a characteristic function, we have 
% $\frac{1}{\sigma^{3/2}}|\na F(\sigma \phi)| \leq 2\frac{1}{\sqrt \sigma}  |\na \phi| |\sigma \phi-\rho|\leq \frac{1}{\sigma \eps}  |\phi-\rho|^2+  \eps |\na \phi|^2$ and so
\begin{equation} \label{eq:BVphi}
\frac{1}{2\sigma^{3/2}} \int_\Omega |\na F(\sigma {\phi^\eps})| \, dx\leq  \J_\eps (\rho^\eps) 
% \qquad \forall \rho \in BV(\Omega;\{0,1\}), \quad\mbox{ where } \phi(x) = \int_\Omega G(x,y) \rho(y)\, dy.
\end{equation}
(note that in \cite{MW} a similar inequality was derived when $\rho$ is a characteristic function. The computation above extends this important property of $\J_\eps$ to the case where $0\leq\rho\leq 1$ by using the formula \eqref{eq:Jepsd2}).

% \begin{equation}\label{eq:young}
% |\na F(\phi)| \leq 2 u v_\eps \leq u_\eps^2 + v_\eps^2 - (u_\eps-v_\eps)^2, \quad \quad  \mbox{ with } 
%  u_\eps : = \eps^{1/2} |\na {\phi^\eps}| \mbox{ and }  v_\eps = \eps^{-1/2} |\phi^\eps-\rho^\eps|.
%\end{equation}
%In particular, we have  $u_\eps^2 + v_\eps^2-|\na F(\phi^\eps)|\geq 0$ and 
%Since $\phi^\eps \to \rho$ in $L^1$ and $F(0)=0$, $F(1) =1/2$, we have $F(\rho^\eps) \to \frac 1 2 \rho$ in $L^1$ and so:
%$$ \liminf_{\eps\to 0} \int_\Omega |\na F(\phi^\eps)| \, dx \geq \frac 1 2 \int_\Omega |\na\rho|$$

Inequality \eqref{eq:BVphi} shows that the boundedness of the energy $\J_{\eps_n}(\rho^{\eps_n }) $ implies some a priori estimates for the auxiliary function
$$ \psi^n := 2 F(\sigma \phi^{\eps_n })$$
More precisely, \eqref{eq:BVphi} and Lemma \ref{lem:unifestimates} (i) imply that
\begin{equation}\label{eq:psiBV}
  \psi^n \mbox{  is bounded in } L^\infty((0,T);BV(\Omega)) .
  \end{equation}

Next, we can write
\begin{align}
\psi^n 
%= [2 F(\sigma \phi^{\eps_n,\tau_n}) - 2 F(\rho^{\eps_n,\tau_n}) ] +[2 F(\rho^{\eps_n,\tau_n})-\rho^{\eps_n,\tau_n}] + \rho^{\eps_n,\tau_n}\nonumber \\
& = [2 F(\sigma \phi^{\eps_n }) - 2 F(\rho^{\eps_n }) ] +[2 F(\rho^{\eps_n })-\rho^{\eps_n}]  + \rho^{\eps_n} \label{eq:psin}
\end{align}
and we are going to show that the first two terms in the right hand side go to zero (uniformly in $t$):
\begin{itemize}
\item Formula \eqref{eq:Jepsd} and the energy bound (Lemma \ref{lem:unifestimates} (i))  imply
$$ \| \rho^{\eps_n}(t)-\sigma \phi^{\eps_n}(t) \|_{L^2(\Omega)}^2 \leq2\sigma \eps_n \J_\eps(\rho^{\eps_n}(t))\leq 2\sigma \eps_n \J_\eps(\rho_{in})\leq C\eps_n$$
and since $F$ is Lipschitz, we deduce
\begin{equation}\label{eq:ghj1}
 \| 2 F(\rho^{\eps_n}(t))-2 F(\sigma \phi^{\eps_n}(t))\|_{L^2(\Omega)}^2 
\leq C  \| \rho^{\eps_n}(t)-\sigma \phi^{\eps_n}(t)\|_{L^2(\Omega)}^2  \leq C\eps_n \qquad \forall t>0.
\end{equation}
\item When $\rho$ is a characteristic function, we have $2 F(\rho) =\rho$ and so the second term in \eqref{eq:psin} vanishes. When $\rho\in (0,1)$, we can use the fact that
$|2 F(\rho) -\rho| \leq C \delta $ whenever $\rho<\delta $ or $\rho>1-\delta$ and use the energy to control the set where $\delta \leq \rho\leq 1-\delta$. Indeed,
formula \eqref{eq:Jepsd} and the energy bound (Lemma \ref{lem:unifestimates} (i))  imply
\begin{equation}\label{eq:rhodelta}
|\{\delta\leq \rho^{\eps_n} \leq 1-\delta\}|\leq \frac{1}{\delta(1-\delta)} \int_\Omega \rho^{\eps_n} (1-\rho^{\eps_n})\, dx \leq  \frac{\sigma \eps_n} {\delta(1-\delta)}  \J_\eps(\rho^{\eps_n}(t))
\leq 
C \frac{ \eps_n}{\delta(1-\delta)}.
\end{equation}
We deduce (with $\delta=\sqrt{ \eps_n}$):
\begin{align*}
\int_\Omega |2 F(\rho^{\eps_n}) -\rho^{\eps_n}| \, dx 
& \leq \int_{\{\sqrt {\eps_n}\leq \rho^{\eps_n}\leq 1-\sqrt {\eps_n}\}} |2 F(\rho^{\eps_n}) -\rho^{\eps_n}| \, dx 
+ C|\Omega|\sqrt {\eps_n}\\
& \leq C |\{\sqrt {\eps_n}\leq \rho^{\eps_n}\leq 1-\sqrt {\eps_n}\}| + C|\Omega|\sqrt {\eps_n}\\
& \leq C\sqrt {\eps_n}
\end{align*}
and so 
\begin{equation}\label{eq:ghj2}
\| 2 F(\rho^{\eps_n(t) }) -\rho^{\eps_n}(t)\|_{L^2(\Omega)} \leq C\eps_n^{1/4} \quad \forall t>0.
\end{equation}
%\item We already saw that $\rho^{\eps_n}-\widetilde\rho^{\eps_n}$ converges to $0$ uniformly in $t$, with respect to $W_2$, and thus also with respect to the $H^{-1}(\Omega)$ norm.
%\item The fact that $\pa_t \widetilde \rho^{\eps_n} = -\div \widetilde E^{\eps_n} $ with $\widetilde E^{\eps_n} $ bounded in $L^2$ implies that $\pa_t \widetilde \phi^{\eps_n}$ is bounded in $L^2(0,T; H^{-1}(\Omega))$. In particular, 
%\item  We already saw that  Lemma \ref{lem:unifestimates} (iii) gives
%that 
\end{itemize}
Since we already know that $  \rho^{\eps_n}$ converges uniformly in $t$, with respect to the  $H^{-1}(\Omega)$ norm, to $\rho$, we deduce from \eqref{eq:psin}, \eqref{eq:ghj1} and \eqref{eq:ghj2} that 
\begin{equation}\label{eq:psiH}
\psi^n(t) \to \rho(t) \mbox{ in $H^{-1}(\Omega)$, uniformly in $t$}.
\end{equation}
\medskip

Using a Lions-Aubin compactness type result (see Lemma \ref{lem:LA}),
\eqref{eq:psiBV} and \eqref{eq:psiH} imply
$$ \psi^n \to \rho \qquad \mbox{ strongly in }  L^\infty((0,T);L^1( \Omega)).$$
In particular, the lower semicontinuity of the $BV$ norm and \eqref{eq:psiBV}  imply that $\rho \in L^\infty((0,T);BV( \Omega))$.
\medskip

Finally, using \eqref{eq:psin} together with 
\eqref{eq:ghj1} and \eqref{eq:ghj2} we see that
$$ \| \rho^{\eps_n}(t) - \psi^n(t)\|_{L^2(\Omega)} \leq C\eps_n^{1/4}$$
so the strong convergence of $\psi^n$ also implies that
%writing $ \rho^{\eps_n} = \psi_n + [2 F(\rho^{\eps_n})-2 F(\sigma \phi^{\eps_n})]$.
%  we get:
$$\rho^{\eps_n} \to \rho \qquad \mbox{ strongly in }  L^\infty((0,T);L^1( \Omega)).$$
\medskip

It remains to show that $\rho$ is a characteristic function. We note that given $t>0$, we can extract a subsequence which converges a.e. in $\Omega$ and \eqref{eq:rhodelta} implies that $|\{\delta\leq \rho(t) \leq 1-\delta\} = 0$ for all $\delta>0$.
We deduce that
$\rho (x,t)\in \{0,1\}$ a.e. $x\in  \Omega$ (for all $t>0$).
\medskip

Finally, we can pass to the limit in  \eqref{eq:energy} to get the energy inequality \eqref{eq:energy2}:
The lower-semicontinuity of $\Theta$ allows us to pass to the limit in the dissipation and 
Proposition \ref{prop:4} gives (since $ \F_\eps (\rho_{in} ) = \J_\eps(\rho_{in}$) when $\rho_{in}$ is a characteristic function)
$$\lim_{\eps\to 0 } \F_\eps (\rho_{in} ) = \J_0(\rho_{in}).$$
The  liminf property of Proposition \ref{prop:Gamma} and the strong convergence of $\rho^\eps$ then yield:
\begin{align*} \liminf_{\eps\to 0 } \F_\eps (\rho_{in} )  
& \geq   \liminf_{\eps\to 0 } \mu \int_\Omega \rho^\eps \log \rho^\eps \, dx + \liminf_{\eps\to 0 } \J_\eps (\rho^\eps)\\
& \geq \mu \int_\Omega \rho  \log \rho \, dx + 
 \J_0(\rho)\\
& \geq  \J_0(\rho).
\end{align*}
This completes the proof of Proposition \ref{prop:rhostrong}.

\end{proof}

\section{Convergence of the first variation}\label{sec:lim}
\subsection{Convergence of the first variation and proof of Theorem \ref{thm:conv1}}
%and 
%$$
%\begin{cases}
%\pa_t \rho  + \div(\sigma^{-3/2}\bar \eps^{-1}\rho \na \bar\phi-\na p)=0 , \mbox{ in } \Omega\times(0,\infty), \qquad p\in P(\rho)\\
%(\eps^{-1}\rho \na \phi-\na p)\cdot n = 0 \mbox{ on }\pa \Omega\times(0,\infty)\\
%\rho(x,0) = \rho_{in}(x) \mbox{ in } \Omega
%\end{cases}
%$$

The only thing left to do to prove Theorem \ref{thm:conv1}
is to pass to the limit in equation \eqref{eq:weak12} (under the convergence assumption \eqref{eq:EA})
to derive \eqref{eq:weak12b}.
We recall \eqref{eq:weak12} here:
$$\int_0^\infty \int_\Omega E^{\eps} \cdot \xi\, dx\, dt   = \int _0^\infty \int_\Omega \eps^{-1} \rho^{\eps} \na \phi^{\eps} \cdot \xi  +\mu  \rho^{\eps}\div \xi+ p^{\eps}\div \xi\, dx \, dt
$$ 
and we note that when $\eps\ll1$, neither the  term $\eps^{-1} \rho^{\eps} \na \phi^{\eps}$ nor  the function $p^{\eps}$ (or  its gradient $\na p^{\eps}$) are bounded. 
As explained in the introduction, it is the modified pressure, defined by
$p^{\eps} + \eps^{-1} \left(\frac 1 {2\sigma} - \phi^{\eps} \right)  \rho^{\eps} $
which is expected to converge to the pressure in the Hele-Shaw model with surface tension. We can also include the term $\mu  \rho^{\eps}$ in this modified pressure and normalize it to have zero average. We thus set:
\begin{equation}\label{eq:modp}
q^{\eps} = p^{\eps} + \eps^{-1} \left(\frac 1 {2\sigma} - \phi^{\eps} \right)  \rho^{\eps} + \mu  \rho^{\eps} + m^\eps \rho^\eps
\end{equation}
where $m^\eps(t)$ (constant in $x$) is chosen so that
\begin{equation}\label{eq:zero}
\int_\Omega q^\eps(x,t)\, dx=0 \quad \forall t>0.
\end{equation}

After a straightforward computation, we can  rewrite  \eqref{eq:weak12}  as 
\begin{equation}\label{eq:weakmc}
\int_0^\infty \int_\Omega E^{\eps} \cdot \xi\, dx\, dt   = 
\int _0^\infty \int_\Omega \eps^{-1} \left(\frac 1 {2\sigma} - \phi^{\eps} \right) \xi  \cdot \na  \rho^{\eps} + q^{\eps}\div \xi\, dx\, dt .
\end{equation}
Passing to the limit in \eqref{eq:weakmc} is the objective of this section. The main challenge being the  term
\begin{equation}\label{eq:sing}
\int _0^\infty \int_\Omega \eps^{-1} \left(\frac 1 {2\sigma} - \phi^{\eps} \right) \xi  \cdot \na  \rho^{\eps}\, dx\, dt,
\end{equation}
which gives rise to the mean-curvature and the contact angle condition in the limit $\eps\to0$.
This term is related to the first variation of the energy $\J_\eps$ defined by \eqref{eq:Jeps}. 
The key result of this section is Proposition \ref{eq:cvass} below, which states that the assumption on the convergence of the energy \eqref{eq:EA} implies the convergence of the first variation. 
Similar results have been proved  for  different energy functionals.
In particular, a classical result of Reshetnyak \cite{Reshetnyak} gives that if  $\chi_{E_\eps}$ converges to $\chi_E$ strongly in $L^1$, then the convergence of the perimeter 
$$ P(E_\eps):=\int|\na \chi_{E_\eps}| \to \int|\na \chi_E|,$$
implies the convergence of the first variation to
$$ \int(\div \xi - \nu \otimes\nu:D\xi) |D\chi_E|.$$
A similar result was proved by Luckhaus and Modica \cite{LM} for the Ginzburg-Landau functional 
$$E^1_\eps (\rho) = \int_\Omega \eps |\na \rho|^2 + \frac 1 \eps (1-\rho^2)^2 \, dx$$
and by Laux and Otto \cite{LO} with
$$ E^1_\eps (\rho)  = \eps^{-1} \int \rho (1- G_\eps \star \rho)\, dx$$
when $G_\eps$ is the Gaussian kernel of variance $\eps^2$.

\medskip

%In our case, we recall that when restricted to characteristic functions, the $\Gamma$-convergence of our energy functional was proved in \cite{MW}.
%More precisely, the sequence of  functionals $\J_\eps$ defined by \eqref{eq:Je1}  $\Gamma$-converges to
%$$
%\J_0(\rho) :=
%\begin{cases}
% \displaystyle \frac 1 {4\sigma^{3/2}} \left[  \int_\Omega |D\rho|  + \int_{\pa \Omega } \min\left( 1 ,\frac{2 \alpha(x)}{\alpha(x)+\sqrt \sigma\beta(x)} \right) \rho(x) \, d\H^{n-1}(x) \right]& \mbox{ if } \rho\in\{0,1\}\\[4pt]
%+\infty & \mbox{ otherwise.}
%\end{cases}
%$$
%\textcolor{red}{We need to comment on the difference between this formula and \eqref{eq:convJ}}.

Crucially, both $E^1$ and $E^2$ are regularizations of the perimeter functional.
In our case, we recall (see Theorem \ref{thm:4}) that the energy functional $\J_\eps$ $\Gamma$-converges to the perimeter functional as well, together with a boundary term.
Furthermore, it is easy to check that \eqref{eq:sing} is the first variation of $\J_\eps(\rho)$ for a perturbation defined by 
\begin{equation}\label{eq:rhos1}
\begin{cases}
\pa_s \rho_s + \na \rho_s \cdot \xi = 0 \\
\rho_s|_{s=0} = \rho.
\end{cases}
\end{equation}
%Indeed, we have:
%$$ 
%\frac{d}{ds} \J_\eps(\rho_s) |_{s=0}  = \frac{1}{2\sigma \eps} \int \pa_s \rho_s|_{s=0} (1-\sigma \phi_s) -  \frac{1}{2  \eps} \int \rho \pa_s \phi_s |_{s=0}.
%$$
%We can then use the facts that $\pa_s \rho_s|_{s=0} = -\na \rho\cdot \xi $ and that  $ \pa_s \phi_s  = \int_\Omega G_\eps(x,y) \pa_s \rho_s\, ds $ to get:
%$$ 
%\delta \J_\eps (\rho)[\xi] =
% \frac{1}{2\sigma \eps}  \int \pa_s \rho_s|_{s=0} (1-\sigma \phi_s) - \frac{1}{2\sigma \eps}  \int\sigma \phi_s   \pa_s  \rho_s |_{s=0}
%=
%-  \eps^{-1} \int\left(\frac{1}{2\sigma}-\phi_s\right)  \na \rho\cdot \xi .
%$$

\begin{remark}
Note that \eqref{eq:rhos1} preserves the constraint $\rho_s\leq1 $ but not the condition $\int \rho_s=1$ (unless $\div \xi=0$).
Alternatively, we could consider the first variation of 
$\J_\eps(\rho)$ for a perturbation defined by 
\begin{equation}\label{eq:rhos2}
\begin{cases}
\pa_s \rho_s + \div( \rho_s \xi) = 0 \\
\rho_s|_{s=0} = \rho.
\end{cases}
\end{equation}
which leads to the integral $  \eps^{-1} \int_\Omega \rho\, \xi \na (\frac{1}{2\sigma}-\phi^\eps) \, dx$. The two integrals are the same when $\div \xi =0 $, but this second integral does not, in general, converge when $\eps\to 0$. This is due to the fact that \eqref{eq:rhos2} does not preserve characteristic functions and the energy $\J_\eps$ blows up in that case for $\eps\ll1$.
%and denote by $\phi_s$ the corresponding solution of \eqref{eq:phi}. We then have
%$$ 
%\delta \J_\eps (\rho)[\xi] = \frac{d}{ds} \J_\eps(\rho_s) |_{s=0}  = \eps^{-1} \int \pa_s \rho_s|_{s=0} (1-\phi^\eps) - \eps^{-1} \int \rho \pa_s \phi_s |_{s=0}.
%$$
%We can now use the fact that $\pa_s \rho_s|_{s=0} = -\na \rho\cdot \xi $ and that the function $\psi = \pa_s \phi_s |_{s=0}$ solves
%$$
%\begin{cases}
%\psi - \eps^2 \Delta \psi = \pa_s\rho_s|_{s=0} = -\na \rho\cdot \xi & \mbox{ in } \Omega\\
%\alpha \psi + \beta \eps \na \psi\cdot n = 0 & \mbox{ on } \Omega
%\end{cases}
%$$
%We can thus write (using \eqref{eq:phi} and a couple of integration by parts):
%$$\int_\Omega \rho \psi \, dx = \int_\Omega [\phi^\eps-\eps^2\Delta {\phi^\eps}]\psi\, dx =  \int_\Omega {\phi^\eps} [\psi-\eps^2\Delta\psi]\, dx = - \int_\Omega {\phi^\eps} \na \rho\cdot \xi\, dx$$
%and \eqref{eq:dJ0}  follows.
%\delta \J_\eps (\rho)[\xi]  = -\eps^{-1} \int_\Omega (1-2\phi^\eps) \xi \cdot \na \rho
\end{remark}

The key result of this section, which will allow us to complete the proof of Theorem \ref{thm:conv1} is the following:
\begin{proposition}\label{prop:firstvar}
Given a sequence of functions $\rho^\eps\in L^1(\Omega)$ satisfying $0\leq \rho^\eps\leq 1$  and such that $\rho^\eps\to\rho$ strongly in $L^1(\Omega)$ with $\rho\in \BV(\Omega;\{0,1\})$
and 
\begin{equation}\label{eq:cvass}
\lim_{\eps\to 0} \J_\eps (\rho^\eps) =\J_0 (\rho),
 \end{equation}
we have, for all $\xi \in C^1(\Omega , \R^d)$ satisfying $\xi\cdot n=0$ on $\pa\Omega$,
\begin{equation}\label{eq:bd} 
\left| \eps^{-1}   \int_\Omega \left(\frac{1}{2\sigma}-\phi^\eps\right) \xi \cdot \na \rho^\eps  \right| \leq C \|D\xi\|_{L^\infty(\Omega)  } \J_\eps(\rho^\eps)
\end{equation}
and
\begin{align}
&\lim_{\eps\to 0} -\eps^{-1} \int_\Omega \left(\frac{1}{2\sigma}-\phi^\eps\right) \xi \cdot \na \rho^\eps
 \nonumber \\
& \qquad \qquad=\frac1 {4\sigma^{3/2}} \left[  \int_\Omega \left[  \div \xi - \nu\otimes \nu :D\xi\right] |\na \rho|
+ \min\left(1 ,\frac{2\alpha}{\alpha+\sqrt\sigma\beta} \right) \int_{\pa\Omega} \left[  \div \xi - n\otimes n :D\xi\right]  \rho\,  d\mathcal H^{n-1}(x)\right]\label{eq:limitdJ} 
\end{align}
where $\nu = \frac{\na \rho}{|\na \rho|}$ and $n$ denotes the outward normal unit vector to the fixed boundary $\pa \Omega $.
\end{proposition}

Note that we can replace \eqref{eq:cvass} with the equivalent condition
$$
\lim_{\eps\to 0} \F_\eps (\rho^\eps) =\F_0 (\rho).
$$
Indeed a bound on $\J_\eps(\rho^\eps)$ or $\F_\eps(\rho^\eps)$ implies that $\rho \in \BV(\Omega;\{0,1\})$ 
(see the proof of Proposition \ref{prop:rhostrong}) so that $\F_0 (\rho) = \J_0 (\rho)$ and $\int \rho^\eps \log \rho^\eps \to 0$.

Proceeding as in \cite{LO}, we can check that this proposition imply the following time-dependent version:
\begin{corollary}\label{cor:1}
Given a sequence of characteristic functions $\rho^\eps(x,t)$ such that $\rho^\eps\to\rho$ in $L^1( \Omega\times(0,T))$ and 
$$
\lim_{\eps\to0} \int_0^T \J_\eps (\rho^\eps)\, dt = \int_0^T \J_0 (\rho) \, dt
$$
we have, for all $\xi \in C^1(\Omega\times(0,T) , \R^d)$ satisfying $\xi\cdot n=0$ on $\pa\Omega$,
\begin{align}
&\lim_{\eps\to 0} -\eps^{-1} \int_0^T \int_\Omega \left(\frac{1}{2\sigma}-\phi^\eps\right) \xi \cdot \na \rho^\eps dt
 \nonumber \\
& \qquad \qquad=\frac1 {4\sigma^{3/2}} \left[ \int_0^T \int_\Omega \left[  \div \xi - \nu\otimes \nu :D\xi\right] |\na \rho| \, dt\right.\nonumber \\
& \qquad \qquad \qquad\left.
+ \min\left(1 ,\frac{2\alpha}{\alpha+\sqrt \sigma\beta} \right) \int_0^T \int_{\pa\Omega} \left[  \div \xi - n\otimes n :D\xi\right]  \rho\,  d\mathcal H^{n-1}(x)\, dt\right]\label{eq:limitdJ2} 
\end{align}

\end{corollary}

With this corollary, we can now  complete the proof of Theorem \ref{thm:conv1}.

\begin{proof}[End of the proof of Theorem \ref{thm:conv1}]
We can now pass to the limit in 
\eqref{eq:weakmc} and thus complete the proof of Theorem \ref{thm:conv1}.
We recall that $E^{\eps_n}$ converges weakly in $L^2$ and the convergence of the first term in the right hand side follows from Corollary \ref{cor:1}, which we can use here since Proposition \ref{prop:rhostrong} gives the strong convergence of $\rho^{\eps_n}$ in $L^1$ and 
we are assuming condition \eqref{eq:EA}.
We thus only need to explain why the pressure $q^\eps$ converges.
\medskip

%This follows from Lemma \ref{lem:firstvar} below.
%Indeed, formula \eqref{eq:dJ} and \eqref{eq:Jepsd} imply
%$$ \left| \eps^{-1} \int_0^T \int_\Omega (1-2\phi^\eps) \xi \cdot \na \rho  \right|\, dt \leq \|D\xi\|_{L^\infty((0,T)\times\Omega)  }\int_0^T \J_\eps(\rho)\, dt
%$$
We note that \eqref{eq:weakmc} and \eqref{eq:bd}  imply
\begin{equation}\label{eq:hsj}
\left| \int_0^T \int_\Omega q^{\eps}\div \xi\, dx\, dt \right|
\leq \| E^{\eps}\|_{L^2(\Omega\times(0,T)) }\| \xi\|_{L^2(\Omega\times(0,T)) }\
+ C \J_\eps(\rho(t)) \|D\xi\|_{L^1((0,T);L^\infty(\Omega))  }
\end{equation}
In  particular, the a priori estimates of Lemma \ref{lem:unifestimates} implies that
 $\na q^{\eps}$ is bounded in $L^2((0,T); (C^1(\Omega))^*)$.

To get a bound on $q^\eps$, we proceed as in  \cite{JKM}: 
For $\vphi \in C^s(\Omega)$, we consider $u$ solution of 
$$ 
\begin{cases}
\Delta u = \vphi  - \fint_\Omega \vphi\, dx & \mbox{ in } \Omega\\
\na u\cdot n = 0 & \mbox{ on } \pa\Omega
\end{cases}
$$
and take $\xi=\na u$ as a test function in \eqref{eq:hsj} to get (using classical Schauder estimates)
$$
\left| \int_0^T \int_\Omega q^{\eps}\Delta u \, dx\, dt \right|
\leq  C \|D^2 u \|_{L^2((0,T); L^\infty(\Omega)) }
\leq C \|u \|_{L^2((0,T); C^{2,s}(\Omega)) }
\leq C \|\vphi \|_{L^2((0,T); C^{s}(\Omega))} 
$$
for any $s>0$.
Using \eqref{eq:zero} we deduce
$$
\left| \int_0^T \int_\Omega q^{\eps}\vphi  \, dx\, dt \right|
\leq C \|\vphi \|_{L^2((0,T); C^{s}(\Omega))} 
$$
which implies that $q^\eps $ is uniformly bounded in $L^2((0,T); (C^s(\Omega))^*)$ and has a weak-* limit $q$.
We can now pass to the limit in \eqref{eq:weakmc}, for $\xi$ smooth enough, and derive \eqref{eq:weak12b}.
 \end{proof}

\subsection{Proof of Proposition \ref{prop:firstvar}}
As noted earlier, if we set $\bar \phi = \sigma \phi $, $\bar \eps= \eps/\sqrt\sigma$ and $\bar \beta =\sqrt \sigma \beta$, equation  \eqref{eq:phi0} becomes
$$
\begin{cases}
\bar \phi -\bar \eps^2\Delta  \bar \phi = \rho  &  \mbox{ in } \Omega\\
 \alpha \bar \phi +\bar \beta \bar \eps \nabla \bar \phi \cdot n = 0 & \mbox{ on } \pa\Omega.
\end{cases}
$$
It is therefore enough to prove the result when $\sigma=1$ and use the fact that
$$
\eps^{-1} \int_\Omega \left(\frac{1}{2\sigma}-\phi^\eps\right) \xi \cdot \na \rho^\eps
={\bar \eps}^{-1} \frac{1}{\sigma^{3/2}} \int_\Omega \left(\frac{1}{2}-\bar {\phi^\eps}\right)  \xi \cdot \na \rho^\eps
$$
to get the result when $\sigma>0$.
\medskip

The proof of Proposition \ref{prop:firstvar} makes use of the following lemma, the proof of which is elementary and presented at the end of this section:
\begin{lemma}\label{lem:firstvar}
Given $\rho(x)$ such that $0\leq\rho\leq 1$ and $\phi$ solution of \eqref{eq:phi0} with $\sigma=1$, 
we have the following formulas, for all $\xi \in C^1(\Omega , \R^d)$:

%let $\rho_s$ be defined by \eqref{eq:rhos1}. Then we have
%\begin{equation}\label{eq:dJ0} 
%\delta \J_\eps (\rho)[\xi] := \frac{d}{dt} \J_\eps(\rho_s)\big|_{s=0}=  -\eps^{-1} \int_\Omega (1-2\phi^\eps) \xi \cdot \na \rho
%\end{equation}
%which can also be written as:
\item(i) If $\alpha=0$ or $\beta=0$ (Neumann or Dirichlet boundary condition for $\phi$) then 
\begin{align}
% \delta \J_\eps (\rho)[\xi] 
 -\eps^{-1} \int_\Omega (1-2\phi ) \xi \cdot \na \rho & =  \eps^{-1} \int_\Omega \left[ (1-\rho  ) \phi ^2 + \rho (1- \phi )^2 \right]\div \xi \, dx + \eps  \int_\Omega   |\na {\phi}|^2 \div \xi\, dx \nonumber \\
&\qquad -2\eps \int_\Omega \na {\phi}\otimes\na\phi : D\xi \, dx  \label{eq:dJN}
\end{align}
\item(ii) If $\beta \neq 0$, then
\begin{align}
% \delta \J_\eps (\rho)[\xi] 
 -\eps^{-1} \int_\Omega (1-2\phi ) \xi \cdot \na \rho & = \eps^{-1} \int_\Omega \left[ (1-\rho  ) \phi ^2 + \rho (1- \phi )^2 \right]\div \xi \, dx  \nonumber\\
 &\qquad\qquad + \eps  \int_\Omega   |\na {\phi }|^2 \div \xi\, dx +  \int_{\pa\Omega}  \frac{\alpha}{\beta}   |\phi |^2 \div \xi \, d\H^{n-1}(x) \nonumber
 \\
& \qquad\qquad 
 -2\eps \int_\Omega \na {\phi }\otimes\na\phi : D\xi \, dx -  \int_{\pa\Omega}\frac{\alpha}{\beta}    {\phi}^2  n\otimes n : D \xi \,   d\H^{n-1}(x).\label{eq:dJ}
\end{align}
\end{lemma} 
{\bf Idea of the proof of Proposition \ref{prop:firstvar}:} 
While the proof of this Proposition \ref{prop:firstvar} may appear long and technical, the idea is quite simple. 
We recall that when   $\alpha=0$ (Neumann boundary condition), we have (see \eqref{eq:Jepsd2}):
\begin{align*}
\J_\eps (\rho^\eps) 
 & = \frac{\eps^{-1}} {2}  \int_\Omega  (1-\rho^\eps ) (\phi^\eps)^2 + \rho^\eps(1- \phi^\eps)^2\  \, dx  + 
\frac{ \eps}{2}  \int_\Omega |\na {\phi^\eps}| ^2\, dx\\
 & =  \frac{ 1} {2} \int _{\Omega}u_\eps^2 + v_\eps^2 \, dx.
 \end{align*}
where we denoted $u_\eps =  \eps^{-1/2} \left[ (1-\rho^\eps ) (\phi^\eps)^2 + \rho^\eps(1- \phi^\eps)^2\right] ^{1/2} $ and $v_\eps = \eps^{1/2}  |\na {\phi^\eps}|$.
The boundedness of the energy implies that $u_\eps$ and $v_\eps$ are bounded in $L^2$ and 
in order to pass to the limit in the first two terms in \eqref{eq:dJN}, we need to show the convergence of 
$ \int [u_\eps^2+v_\eps^2] \div\xi\, dx$.

Using the notations of Proposition \ref{prop:rhostrong}, and in particular the function $F$ such that $F'(\phi) = 2 \min(\phi,1-\phi)$, we now write
\begin{equation}\label{eq:kjhkhjb} 
|\na F(\phi^\eps)| =  2 \min(\phi^\eps,1-\phi^\eps)|\na {\phi^\eps}| \leq 2 u_\eps v_\eps \leq u_\eps^2 + v_\eps^2
\end{equation}
Since the limit  $\rho$ of $\rho^\eps$ is a characteristic function and we know that $\phi^\eps \to \rho $, 
we have $F(\phi^\eps) \to F(\rho) = \frac 1 2 \rho$ in $L^1$ (see the proof of Proposition \ref{prop:rhostrong}), and so 
$$\liminf_{\eps\to 0}  \int_\Omega  |\na F(\phi^\eps)| \, dx \geq \frac{1}{2} \int_\Omega |\na \rho|.$$
On the other hand, the assumption of convergence of the energy, \eqref{eq:cvass}, implies
$$ \lim_{\eps\to 0} \int _{\Omega}u_\eps^2 + v_\eps^2 \, dx = 2 \J_0(\rho) = \frac 1 2 \int |\na \rho|.$$
Together, these inequalities imply that there is equality in \eqref{eq:kjhkhjb} 
when $\eps\to0$, which means that $ u_\eps^2 + v_\eps^2 -|\na F(\phi^\eps)|  \to 0$ in $L^1$
 and that $\lim_{\eps\to0} \int_\Omega  |\na F(\phi^\eps)| \, dx = \frac 1 2 \int |\na \rho|$.
A Classical result (see Proposition \ref{prop:BV}) now implies that
$$ \lim_{\eps\to 0 } \int_\Omega [u_\eps^2+v_\eps^2] \div\xi\, dx = \lim_{\eps\to 0 } \int_\Omega |\na F(\phi^\eps)| \div\xi\, dx=  \frac 1 2 \int_\Omega |\na \rho| \div \xi.$$

%\textcolor{blue}{ and $\frac{\eps^{-1}} {2}  \int_\Omega \rho^\eps (1-\rho^\eps)  \, dx \to 0$ which implies that $\frac{\eps^{-1}} {2}  \int_\Omega \rho^\eps (1-\rho^\eps) \div \xi \, dx \to0$}

To pass to the limit in the last term in \eqref{eq:dJN}, we note that the (asymptotic) equality in Young's inequality in 
\eqref{eq:kjhkhjb}  also implies that $u_\eps- v_\eps \to 0$ in $L^2$ and so $2\eps^{-1} |\na {\phi^\eps}|^2 = 2 v_\eps^2\sim u_\eps^2 + v_\eps^2 \sim |\na F(\phi^\eps)|$
which proves the convergence of $2\eps^{-1} |\na {\phi^\eps}|^2 $.
A simple (if somewhat technical) result (see Proposition   \ref{prop:measure}) then shows that the convergence of $\eps^{-1} |\na {\phi^\eps}|^2$ implies that of $\eps^{-1} \na {\phi^\eps}\otimes\na {\phi^\eps}$.

Additional care will be needed to take care of the boundary condition when $\alpha\neq 0$, which is why we will first give the detailed proof for Neumann boundary conditions, then Dirichlet conditions (which requires extending $\phi^\eps$ to $\R^d$ by $0$) and finally general Robin boundary conditions (which combine the Neumann and Dirichlet case).

%\begin{lemma}\label{lem:eta}
%Given a measure $\mu$ and functions $g$, $h$ in $L^2(d\mu)$ and $f\in L^1(d\mu)$, if
%$ 0\leq f \leq 2gh $ and $\int g^2+h^2\, d\mu \leq \int f d\mu + \eta$, then
%$$
%\int |f - (g^2+h^2)| d\mu  \leq  
%  \eta
%$$
%and 
%$$
% \left| \int f d\mu - 2 \int g^2\, d\mu\right| \leq C \sqrt{\| f\|_{L^1(d\mu)}} \sqrt \eta
%$$
%\end{lemma}
%When $\eta=0$, the result follows from the equality in Young's inequality $2gh\leq g^2+h^2$.   

%Note in particular that \eqref{lem:firstvar} implies the following result which was used to prove the boundedness of the pressure $\na q^{\eps}$
%\begin{corollary}\lablel{eq:bhf}
%Given $\rho$ a characteristic function and $\phi$ solution of \eqref{eq:phi0} we have
%$$
%\left|
%\eps^{-1} \int_\Omega (1-2\phi^\eps) \xi \cdot \na \rho \right|
%\leq C \J_\eps(\rho) \| D\xi\|_{L^\infty(\Omega)}
%$$
%\end{corollary}

\begin{proof}[Proof of Proposition \ref{prop:firstvar}]
We note that \eqref{eq:dJ} together with \eqref{eq:Jepsd} immediately imply  \eqref{eq:bd}.

The difficult part of the proof is, of course, to establish the limit \eqref{eq:limitdJ}, and we will first give the proof in the simpler case of {\bf Neumann boundary conditions} ($\alpha=0$). We have to pass to the limit in \eqref{eq:dJN}.
%Following the proof of the $\Gamma$ convergence of $\J_\eps$ from \cite{MW}, we
% introduce the function
%$$ F(t) = \int_0^t 2 \min(\tau,1-\tau)\, d\tau = \begin{cases} t^2 & \mbox{ for } 0\leq t\leq 1/2 \\ 2t-t^2 -\frac 1 2 & \mbox{ for } 1/2 \leq t\leq 1\end{cases} .$$
As above, we use the function $F$ such that $F'(\phi) = 2 \min(\phi,1-\phi)$ and note that 
 $$F'(\phi)\leq 2 \left[ (1-\rho ) \phi ^2 + \rho (1- \phi)^2\right] ^{1/2} 
 $$ 
when $0\leq \rho\leq 1$.
 We thus have 
 $|\na F(\phi^\eps)| \leq 2 |\na {\phi^\eps}| \left[ (1-\rho ^\eps) (\phi^\eps) ^2 + \rho^\eps (1- \phi^\eps)^2\right] ^{1/2}  $ and so
 \begin{equation}\label{eq:young}
 |\na F(\phi^\eps)| \leq 2 u_\eps v_\eps \leq u_\eps^2 + v_\eps^2 - (u_\eps-v_\eps)^2,
\end{equation}
where 
$$
  u_\eps : =  \eps^{-1/2} \left[ (1-\rho ^\eps) (\phi^\eps) ^2 + \rho^\eps (1- \phi^\eps)^2\right] ^{1/2}\quad  \mbox{ and } \quad  v_\eps =\eps^{1/2} |\na {\phi^\eps}| .
$$
%In particular, we have  $u_\eps^2 + v_\eps^2-|\na F(\phi^\eps)|\geq 0$ and 

Next, the strong convergence of $\rho^\eps$ and \eqref{eq:ghj1} imply that $F(\phi^\eps) $ converges to $F(\rho)$ strongly in $L^1$. Since $\rho$ is a characteristic function and $F(0)=0$, $F(1) =1/2$, we deduce $F(\phi^\eps) \to F(\rho)= \frac 1 2 \rho$ strongly  in $L^1$ and so:
$$ \liminf_{\eps\to 0} \int_\Omega |\na F(\phi^\eps)| \, dx \geq \frac 1 2 \int_\Omega |\na\rho|.$$
On the other hand,  the convergence assumption \eqref{eq:cvass} implies 
$$ \int_\Omega  u_\eps^2 + v_\eps^2\, dx = 2 \J_\eps (\rho^\eps) \to \frac 1 2 \int_\Omega |\na\rho|.$$
Inequality \eqref{eq:young} thus implies:
\begin{align}
& u_\eps^2 + v_\eps^2 - |\na F(\phi^\eps)| \to 0 \quad \mbox{ in } L^1(\Omega) \label{eq:conv2} \\
&  \int_\Omega   |\na F(\phi^\eps)| \, dx  \to \frac 1 2 \int_\Omega |\na\rho| \label{eq:conv1}\\
& u_\eps-v_\eps \to 0 \quad \mbox{ in } L^2(\Omega). \label{eq:conv3}
\end{align}
These facts  allow us to pass to the limit in the first two terms of \eqref{eq:dJN}.
Indeed, using first the definition of $u^\eps$ and $v^\eps$, then the limit \eqref{eq:conv2} and finally  \eqref{eq:conv1} (together with Proposition \ref{prop:BV}), we can write:
\begin{align}
\lim_{\eps\to 0} 
\eps^{-1} \int_\Omega (\rho^\eps-\phi^\eps)^2  \div \xi \, dx + \eps  \int_\Omega   |\na {\phi^\eps}|^2 \div \xi\, dx 
&  =\lim_{\eps\to 0}  \int_\Omega ( u_\eps^2 + v_\eps^2) \div \xi \, dx \nonumber\\
&  =\lim_{\eps\to 0}  \int_\Omega |\na F(\phi^\eps)|  \div \xi \, dx\nonumber \\
 &= \frac 1 2 \int_\Omega \div \xi |\na\rho| .\label{eq:lim1}
\end{align}
Furthermore, \eqref{eq:conv2} and \eqref{eq:conv3} yields:
\begin{equation}\label{eq:conv4}
2u_\eps^2 - |\na F(\phi^\eps)| \to 0 \quad \mbox{ in } L^1(\Omega) 
\end{equation}
which we use to pass to the limit in the term involving $\na {\phi^\eps}\otimes\na\phi^\eps$. Indeed, we can write
$$
2\eps \int_\Omega \pa_i {\phi^\eps} \pa_j {\phi^\eps} \pa_i\xi_j \, dx
 = 2\eps \int_\Omega \frac{\pa_i {\phi^\eps}}{|\na {\phi^\eps}|}\frac{ \pa_j {\phi^\eps} }{|\na {\phi^\eps}|} \pa_i\xi_j \, {|\na {\phi^\eps}|^2}dx
 =  \int_\Omega \frac{\pa_i {\phi^\eps}}{|\na {\phi^\eps}|}\frac{ \pa_j {\phi^\eps} }{|\na {\phi^\eps}|} \pa_i\xi_j \, 2u_\eps^2 dx
$$ 
and since $\frac{\pa_i {\phi^\eps}}{|\na {\phi^\eps}|}\frac{ \pa_j {\phi^\eps} }{|\na {\phi^\eps}|} \pa_i\xi_j $ is bounded in $L^\infty$, \eqref{eq:conv4} implies that
$$
\lim_{\eps\to 0} 
2\eps \int_\Omega \pa_i {\phi^\eps} \pa_j {\phi^\eps} \pa_i\xi_j \, dx
=
\lim_{\eps\to 0} 
 \int_\Omega \frac{\pa_i {\phi^\eps}}{|\na {\phi^\eps}|}\frac{ \pa_j {\phi^\eps} }{|\na {\phi^\eps}|} \pa_i\xi_j \, |\na F(\phi^\eps)|  dx.
$$
Using the fact that $F'(\phi)\geq 0$, we can also write
$$
\lim_{\eps\to 0} 
2\eps \int_\Omega \pa_i {\phi^\eps} \pa_j {\phi^\eps} \pa_i\xi_j \, dx
=
\lim_{\eps\to 0} 
 \int_\Omega \frac{\pa_i F( {\phi^\eps})}{|\na F( {\phi^\eps})|}\frac{ \pa_j F(\phi^\eps) }{|\na F(\phi^\eps)|} \pa_i\xi_j \, |\na F(\phi^\eps)|  dx.
$$
and using \eqref{eq:conv1} and Proposition \ref{prop:measure} we deduce
\begin{equation}\label{eq:lim2}
\lim_{\eps\to 0} 
2\eps \int_\Omega \pa_i {\phi^\eps} \pa_j {\phi^\eps} \pa_i\xi_j \, dx
=\frac 1 2  \int_\Omega \frac{\pa_i \rho }{|\na \rho|}\frac{ \pa_j \rho }{|\na \rho |} \pa_i\xi_j \, |\na \rho|  
\end{equation}
   
\medskip
Using \eqref{eq:dJN} together with \eqref{eq:lim1} and \eqref{eq:lim2} we get
\begin{align*}
\lim_{\eps\to 0}
 -\eps^{-1} \int_\Omega (1-2\phi^\eps) \xi \cdot \na \rho  %= \lim_{\eps\to 0}\left[ \eps^{-1} \int_\Omega (\rho-\phi^\eps)^2  \div \xi \, dx + \eps  \int_\Omega   |\na {\phi^\eps}|^2 \div \xi\, dx  
% -2\eps \int_\Omega \na {\phi^\eps}\otimes\na\phi^\eps : D\xi \, dx  \right]\\
 & 
 = \frac 1 2 \int_\Omega \div \xi |\na\rho| 
   +\frac 1 2  \int_\Omega \frac{\pa_i \rho }{|\na \rho|}\frac{ \pa_j \rho }{|\na \rho |} \pa_i\xi_j \, |\na \rho|  
   \end{align*}
   which gives \eqref{eq:limitdJ} in the case of Neumann boundary condition $\alpha=0$ (and when $\sigma=1$) 
   
\medskip

For the case of {\bf Dirichlet conditions} ($\beta=0$), we proceed similarly, but we first extend $\phi^\eps$ to $\R^d$ by setting it equal to $0$ in $\R^d\setminus\Omega$. Denoting $\bar {\phi^\eps}$ this extension, we find (using 
inequality \eqref{eq:young} and the Dirichlet boundary condition for $\phi^\eps$):
% \begin{equation}\label{eq:young3}
% |\na (F(\bar {\phi^\eps})| \leq 2 \bar u_\eps \bar v_\eps \leq \bar u_\eps^2 + \bar v_\eps^2, \mbox{ in } \R^d \qquad \quad  \mbox{ where } 
%  \bar u_\eps : = \eps^{1/2} |\na \bar {\phi^\eps}| \mbox{ and }  \bar v_\eps = \eps^{-1/2} |\bar {\phi^\eps}-\bar \rho^\eps|
%\end{equation}
%and using the Dirichlet boundary condition for $\phi^\eps$, we get
$$\int_{\R^d} |\na F(\bar {\phi^\eps})| \, dx = \int_{\Omega} |\na F( {\phi^\eps})| \, dx 
\leq %\int_{\R^d}  \bar u_\eps^2 + \bar v_\eps^2\, dx =
 \int_\Omega u_\eps^2 +   v_\eps^2\, dx =2\J_\eps(\rho^\eps).$$
The lower semicontinuity of the BV norm and assumption \eqref{eq:cvass} give
$$\liminf_{\eps\to 0} \int_{\R^d} |\na (F(\bar {\phi^\eps})| \, dx \geq \int_{\R^d} |\na F(\bar \rho)| =\frac 1 2 \int_{\R^d} |\na \bar \rho| 
 %=\frac 1 2 \left( \int_\Omega |\na \rho|  + \int_{\pa\Omega} \rho d\H^{n-1}(x) \right) 
 = \lim_{\eps\to 0 } 2\J_\eps(\rho^\eps) .
 $$
 Indeed, when $\beta=0$ (and $\sigma=1$), we have
 $$\J_0(\rho) = \frac{1}{4}\left[ \int_\Omega |\na \rho| + \int_{\pa \Omega} \rho(x) d\H^{n-1}(x)\right] = \frac 1 4  \int_{\R^d} |\na \bar \rho| .
$$
 We can thus proceed as before to show that
 \begin{align*}
& u_\eps^2 + v_\eps^2 - |\na F( {\phi^\eps})| \to 0 \quad \mbox{ in } L^1(\Omega) \\
&  \int_{\R^d}   |\na F(\bar {\phi^\eps})| \, dx  \to \frac 1 2 \int_{\R^d} |\na\rho| \\
&   u_\eps-  v_\eps \to 0 \quad \mbox{ in } L^2(\Omega). 
\end{align*}
These are the same convergences as \eqref{eq:conv2}-\eqref{eq:conv3}, except for \eqref{eq:conv1} which involves the extension $\bar\phi^\eps$.
We can now write:
\begin{align*}
\lim_{\eps\to0} 
& \eps^{-1} \int_\Omega \left[ (1-\rho ^\eps ) ( \phi^\eps) ^2 + \rho ^\eps(1- \phi^\eps )^2 \right] \div \xi \, dx + \eps  \int_\Omega   |\na {\phi^\eps}|^2 \div \xi\, dx \\
% &=\eps^{-1} \int_{\R^d} (\bar \rho^\eps-\bar {\phi^\eps})^2  \div \xi \, dx + \eps  \int_{\R^d}   |\na \bar {\phi^\eps}|^2 \div \xi\, dx \\
&\qquad\qquad =\lim_{\eps\to0}  \int_{\Omega} \div \xi  (  u_\eps^2 +  v_\eps^2) \, dx\\
&\qquad\qquad =\lim_{\eps\to0}  \int_{\Omega}\div \xi   |\na F( {\phi^\eps})| \, dx \\
&\qquad \qquad=\lim_{\eps\to0}  \int_{\R^d}\div \xi   |\na F( \bar {\phi^\eps})| \, dx\\
 &\qquad \qquad= \frac 1 2 \int_{\R^d} \div \xi |\na\rho|  = \frac 1 2 \left( \int_{\Omega} \div \xi |\na\rho|  +\int_{\pa\Omega}\rho \,\div \xi \, d\H^{n-1}(x) \right) .
\end{align*}
Similarly, we can show (as before)$$
2\eps \int_\Omega \na {\phi^\eps}\otimes\na\phi^\eps : D\xi \, dx
\longrightarrow
\frac 1 2 \int_{\R^d} \left( \frac{\na \bar \rho}{|\na \bar\rho|}\otimes\frac{\na\bar \rho}{|\na\bar \rho|} : D\xi \right) |\na \bar\rho|$$
where
$$
 \int_{\R^d}\left( \frac{\na\bar \rho}{|\na \bar\rho|}\otimes\frac{\na\bar \rho}{|\na \bar\rho|} : D\xi \right) |\na \bar \rho|
=  \int_{\Omega} \left( \nu \otimes\nu : D\xi \right)   |\na \rho|
+
 \int_{\pa\Omega} \left( n\otimes n : D\xi \right) \rho \, d\H^{n-1}(x).$$
This gives the result when $\beta=0$ (Dirichlet boundary conditions).

\medskip
\medskip

For the case of {\bf Robin conditions}, we need to combine the two cases above.
We first choose  $\vphi:\R^d \to [0,1]$ continuous function such that 
$$  \int_{\pa\Omega }|\chi_E- \vphi |\, d\H^{n-1}\leq \delta.$$
%We then  write, as in the Neumann boundary case
%$$ |\na F(\phi^\eps)| (1-\vphi) \leq 2 u_\eps v_\eps (1-\vphi) \leq (u_\eps^2 + v_\eps^2 - (u_\eps-v_\eps)^2 ) (1-\vphi) .$$
%and (as in the Dirichlet case)
%$$
%Since $\alpha/\eta\geq 0$, we can:
%\begin{align*}
%\J_\eps (\rho^\eps) 
% & =  \eps^{-1}\int _{\Omega}(\rho^\eps-\phi^\eps)^2 \, dx  + 
% \eps  \int_\Omega |\na {\phi^\eps}| ^2\, dx + \int_{\pa \Omega}\frac{\alpha}{\beta} |\phi^\eps|^2 \, d\H^{n-1}(x) \\
% & \geq   \eps^{-1}\int _{\Omega}(\rho^\eps-\phi^\eps)^2 \, dx  + 
% \eps  \int_\Omega |\na {\phi^\eps}| ^2\, dx \\
%& \geq  \int_\Omega |\na F(\phi^\eps)| \, dx
% \end{align*}
Introducing the function $G(t) = \frac \alpha \beta t^2 - F(t)$, we then write:
\begin{align*}
 \int_{\R^d} |\na F(\overline{\phi^\eps})| \vphi \, dx +  \int_{\pa\Omega }G(\phi^\eps)  \vphi d\H^{n-1}
& =  \int_\Omega |\na F(\phi^\eps)| \vphi \, dx + \int_{\pa\Omega } F(\phi^\eps) \vphi d\H^{n-1} +  \int_{\pa\Omega }G(\phi^\eps)  \vphi d\H^{n-1} \\
& =  \int_\Omega |\na F(\phi^\eps)| \vphi \, dx + \int_{\pa\Omega }  \frac \alpha \beta  |\phi^\eps|^2 \vphi d\H^{n-1} 
\end{align*}
which leads to
\begin{align}
&  \int_{\Omega} |\na F( {\phi^\eps})| (1-\vphi ) \, dx + \int_{\R^d} |\na F(\overline{\phi^\eps})| \vphi \, dx +  \int_{\pa\Omega }G(\phi^\eps)  \vphi d\H^{n-1} +\int_{\pa\Omega }  \frac \alpha \beta  |\phi^\eps|^2 (1-\vphi) d\H^{n-1} \nonumber \\
 &\qquad =  \int_\Omega |\na F(\phi^\eps)| \, dx + \int_{\pa\Omega }  \frac \alpha \beta  |\phi^\eps|^2  d\H^{n-1}.\label{eq:fdsfnkj}
 \end{align}
 The right hand side satisfies (with the same notations as above):
 \begin{align}
\int_\Omega |\na F(\phi^\eps)| \, dx + \int_{\pa\Omega }  \frac \alpha \beta  |\phi^\eps|^2  d\H^{n-1}
 & \leq   \int_\Omega u_\eps^2 + v_\eps^2 - (u_\eps-v_\eps)^2    \, dx + \int_{\pa\Omega }  \frac \alpha \beta  |\phi^\eps|^2  d\H^{n-1} \nonumber \\
& = 2\J_\eps(\rho^\eps)-  \int_\Omega  (u_\eps-v_\eps)^2    \, dx \label{eq:bhkj}
\end{align}
so in order to proceed as before, we need to show that the $\liminf$ of the left hand side is greater than  $2 \J_0(\rho)$.
For this, we notice that the function
%On the other hand, we can use the extension of $\overline{\phi^\eps}$ by $0$ as in the Dirichlet case:
%We introduce the function
$$G(t) = \frac \alpha \beta t^2 - F(t)= \begin{cases} (\frac \alpha \beta -1 )t^2 & \mbox{ for } 0\leq t\leq 1/2 \\ (\frac \alpha \beta +1) t^2-2t+ \frac 1 2 & \mbox{ for } 1/2 \leq t\leq 1\end{cases} $$  
satisfies
$G(t) \geq \min \{ 0, \frac \alpha{\alpha+\beta}-\frac1  2\}$ for all $t\in[0,1]$, so that we can write
\begin{align}
& \liminf_{\eps\to 0}
 \int_{\Omega} |\na F( {\phi^\eps})| (1-\vphi ) \, dx +  \int_{\R^d} |\na F(\overline{\phi^\eps})| \vphi \, dx +  \int_{\pa\Omega }G(\phi^\eps)  \vphi d\H^{n-1}+\int_{\pa\Omega }  \frac \alpha \beta  |\phi^\eps|^2 (1-\vphi) d\H^{n-1}  \nonumber \\
& \qquad\geq  \int_\Omega |\na F(\rho)| (1-\vphi)\, dx +  \int_{\R^d} |\na F(\bar \rho) | \vphi \, dx +  \int_{\pa\Omega }  \min \left\{0, \frac \alpha{\alpha+\beta}-\frac1  2\right\} \vphi d\H^{n-1} +0\nonumber \\
&\qquad \geq 
 \int_\Omega\frac 1 2 |\na \rho | \, dx +  \int_{\pa\Omega }\frac 1 2 \rho \vphi \, d\H^{n-1} +  \int_{\pa\Omega }  \min \left\{0, \frac \alpha{\alpha+\beta}-\frac1  2\right\}\vphi d\H^{n-1}\nonumber  \\
&\qquad\geq 
 \int_\Omega\frac 1 2 |\na \rho | \, dx +  \int_{\pa\Omega }  \min \left\{ \frac 1 2, \frac \alpha{\alpha+\beta}\right\}\rho  d\H^{n-1} -C\delta \nonumber 
\\
& \qquad  = 2 \J_0(\rho) -C\delta \label{eq:liminfgh}
\end{align}
thanks to our particular choice of $\vphi$.
Going back to \eqref{eq:fdsfnkj}, we see that we just showed that
$$\liminf_{\eps\to 0} \int_\Omega |\na F(\phi^\eps)| \, dx + \int_{\pa\Omega }  \frac \alpha \beta  |\phi^\eps|^2  d\H^{n-1} \geq  
\J_0(\rho) -C\delta 
$$
and since this holds for any $\delta>0$, we get
$$\liminf_{\eps\to 0} 
\int_\Omega |\na F(\phi^\eps)| \, dx + \int_{\pa\Omega }  \frac \alpha \beta  |\phi^\eps|^2  d\H^{n-1} \geq  
\J_0(\rho).
$$

We can now conclude as in the previous cases:
Using \eqref{eq:bhkj} and the assumption that $\lim_{\eps\to 0}  \J_\eps(\rho^\eps) = \J_0(\rho)$ to conclude that
\begin{align}
& u_\eps^2 + v_\eps^2 - |\na F(\phi^\eps)| \to 0 \quad \mbox{ in } L^1(\Omega) \label{eq:conv23} \\
&  \int_\Omega   |\na F(\phi^\eps)| \, dx  \to \frac 1 2 \int_\Omega |\na\rho| \label{eq:conv13}\\
& u_\eps-v_\eps \to 0 \quad \mbox{ in } L^2(\Omega). \label{eq:conv33}
\end{align}
Furthermore, using \eqref{eq:liminfgh}, we also get
$$ \limsup_{\eps\to0} \left|  \int_{\Omega} |\na F( {\phi^\eps})| (1-\vphi ) \, dx -  \frac 1 2 \int_\Omega |\na \rho| (1-\vphi) \right|\leq  C\delta,$$
$$ \limsup_{\eps\to0} \left| \int_{\R^d} |\na F(\overline{\phi^\eps})| \vphi \, dx - \frac 1 2  \int_{\R^d}|\na  \rho |\vphi  \right|\leq  C\delta $$
$$ \limsup_{\eps\to0} \left| 
\int_{\pa\Omega }G(\phi^\eps)  \vphi d\H^{n-1}
- \int_{\pa\Omega }  \min \left\{0, \frac \alpha{\alpha+\beta}-\frac1  2\right\} \vphi d\H^{n-1} 
 \right|\leq  C\delta $$
and 
$$
 \limsup_{\eps\to0}  
\int_{\pa\Omega }  \frac \alpha \beta  |\phi^\eps|^2 (1-\vphi) d\H^{n-1}  \leq  C\delta .
 $$
We then write (using \eqref{eq:conv23})
\begin{align}
 & \lim_{\eps\to0} \eps^{-1} \int_\Omega [(1-\rho^\eps)(\phi^\eps)^2 + \rho^\eps (1-\phi^\eps)^2] \div \xi \, dx + \eps  \int_\Omega   |\na {\phi^\eps}|^2 \div \xi\, dx +  \int_{\pa\Omega}  \frac{\alpha}{\beta}   |\phi^\eps|^2 \div \xi \, d\H^{n-1}(x)\nonumber \\
% &=\eps^{-1} \int_{\R^d} (\bar \rho^\eps-\bar {\phi^\eps})^2  \div \xi \, dx + \eps  \int_{\R^d}   |\na \bar {\phi^\eps}|^2 \div \xi\, dx \\
& \qquad =\lim_{\eps\to0}  \int_{\Omega}   (  u_\eps^2 +  v_\eps^2)\div \xi \, dx+  \int_{\pa\Omega}  \frac{\alpha}{\beta}   |\phi^\eps|^2 \div \xi \, d\H^{n-1}(x)\nonumber \\
& \qquad =\lim_{\eps\to0}  \int_{\Omega} |\na F( {\phi^\eps})| \div \xi  \, dx +  \int_{\pa\Omega}  \frac{\alpha}{\beta}   |\phi^\eps|^2 \div \xi \, d\H^{n-1}(x)\label{eq:jhkg}
\end{align}
and using the same function $\vphi$ as above, we decompose the integral in \eqref{eq:jhkg}
 as follows:
\begin{align*}
&  \int_{\Omega} |\na F( {\phi^\eps})|\div \xi   \, dx +  \int_{\pa\Omega}  \frac{\alpha}{\beta}   |\phi^\eps|^2 \div \xi \, d\H^{n-1}(x)\\
& \qquad =   \int_{\Omega} |\na F( {\phi^\eps})|   \vphi \div \xi\, dx +  \int_{\pa\Omega}  \frac{\alpha}{\beta} |\phi^\eps|^2 \vphi   \div \xi \, d\H^{n-1}(x)\\
&\qquad\qquad + \int_{\Omega} |\na F( {\phi^\eps})|(1-\vphi)  \div \xi \, dx +  \int_{\pa\Omega}  \frac{\alpha}{\beta} |\phi^\eps|^2 (1-\vphi)  \div \xi \, d\H^{n-1}(x)\\
& \qquad =   \int_{\R^d}|\na F(\overline  {\phi^\eps})|  \vphi \div \xi \, dx +  \int_{\pa\Omega} G(\phi^\eps) \vphi  \div \xi \, d\H^{n-1}(x)\\
&\qquad\qquad + \int_{\Omega} |\na F( {\phi^\eps})|(1-\vphi)  \div \xi \, dx +  \int_{\pa\Omega}  \frac{\alpha}{\beta}  |\phi^\eps|^2 (1-\vphi) \div \xi \, d\H^{n-1}(x)
\end{align*}
We deduce
\begin{align*}
\limsup_{\eps\to0}
& \left| 
  \int_{\Omega} |\na F( {\phi^\eps})|\div \xi   \, dx +  \int_{\pa\Omega}  \frac{\alpha}{\beta}   |\phi^\eps|^2 \div \xi \, d\H^{n-1}(x) \right. \\
&   \quad \left.
  -  
   \frac 1 2 \int_\Omega |\na \rho| (1-\vphi) \div \xi-\frac 1 2  \int_{\R^d} |\na \rho| \vphi  \div \xi
-\int_{\pa\Omega }  \min \left\{0, \frac \alpha{\alpha+\beta}-\frac1  2\right\} \vphi \div \xi d\H^{n-1} \right|\leq C\delta
\end{align*}
that is
$$
\limsup_{\eps\to0}
\left| 
  \int_{\Omega} |\na F( {\phi^\eps})|\div \xi   \, dx +  \int_{\pa\Omega}  \frac{\alpha}{\beta}   |\phi^\eps|^2 \div \xi \, d\H^{n-1}(x)
   -\frac 1 2  \int_{\Omega} |\na \rho| \div \xi
-\int_{\pa\Omega }  \min \left\{\frac 1 2, \frac \alpha{\alpha+\beta} \right\} \vphi \div \xi d\H^{n-1} \right|\leq C\delta
$$
and the choice of $\vphi$ implies
$$
\limsup_{\eps\to0}
\left| 
  \int_{\Omega} |\na F( {\phi^\eps})|\div \xi   \, dx +  \int_{\pa\Omega}  \frac{\alpha}{\beta}   |\phi^\eps|^2 \div \xi \, d\H^{n-1}(x)
   -\frac 1 2  \int_{\Omega} |\na \rho| \div \xi
-\int_{\pa\Omega }  \min \left\{\frac 1 2, \frac \alpha{\alpha+\beta} \right\} \chi_E \div \xi d\H^{n-1} \right|\leq C\delta.
$$
Since the left hand side is independent of $\vphi$, we can take $\delta\to0$ and use \eqref{eq:jhkg} to get
\begin{align*}
& \lim_{\eps\to 0} \eps^{-1} \int_\Omega  [(1-\rho^\eps)(\phi^\eps)^2 + \rho^\eps (1-\phi^\eps)^2]  \div \xi \, dx + \eps  \int_\Omega   |\na {\phi^\eps}|^2 \div \xi\, dx +  \int_{\pa\Omega}  \frac{\alpha}{\beta}   |\phi^\eps|^2 \div \xi \, d\H^{n-1}(x) \\
& \qquad =
  \frac 1 2  \int_{\Omega} |\na \rho | \div \xi
+\int_{\pa\Omega }  \min \left\{\frac 1 2, \frac \alpha{\alpha+\beta} \right\} \rho \div \xi d\H^{n-1} 
\end{align*}

It then only remains to show that
\begin{align*}
& \lim_{\eps\to 0} 2\eps \int_\Omega \na {\phi^\eps}\otimes\na\phi^\eps : D\xi \, dx +  \int_{\pa\Omega}\frac{\alpha}{\beta}    {\phi^\eps}^2  n\otimes n : D \xi \,   d\H^{n-1}(x)\\
& \qquad =\frac1 2  \int_\Omega  \nu\otimes \nu :D\xi|\na \rho|
+ \min\left(\frac 1 2 ,\frac{\alpha}{\alpha+ \beta} \right) \int_{\pa\Omega}  n\otimes n :D\xi\  \rho\,  d\mathcal H^{n-1}(x)
\end{align*}
which can be easily done by combining the argument above with how that term was handled in the case of Dirichlet boundary conditions.

\medskip

\vspace{20pt}

\end{proof}

\begin{proof}[Proof of Lemma \ref{lem:firstvar}]
We write (using  the fact that $\xi\cdot n=0$ on $\pa\Omega$ and \eqref{eq:phi0}):
\begin{align*}
 -\eps^{-1} \int_\Omega (1-2\phi^\eps) \xi \cdot \na \rho&  = \eps^{-1} \int_\Omega \rho (1-2\phi^\eps) \div \xi \, dx - 2 \eps^{-1}  \int_\Omega \rho \na\phi^\eps\cdot\xi\,dx\\
& = \eps^{-1} \int_\Omega \rho (1-2\phi^\eps) \div \xi \, dx - 2 \eps^{-1}  \int_\Omega {\phi^\eps} \na\phi^\eps\cdot\xi\,dx
+2 \eps  \int_\Omega \Delta {\phi^\eps} \na\phi^\eps\cdot\xi\,dx\\
& = \eps^{-1} \int_\Omega \rho (1-2\phi^\eps) \div \xi \, dx + \eps^{-1}  \int_\Omega {\phi^\eps} ^2  \div \xi\,dx
+2 \eps  \int_\Omega \Delta {\phi^\eps} \na\phi^\eps\cdot\xi\,dx
\\
& = \eps^{-1} \int_\Omega [\rho -2\rho {\phi^\eps}+(\phi^\eps)^2] \div \xi \, dx+2 \eps  \int_\Omega \Delta {\phi^\eps} \na\phi^\eps\cdot\xi\,dx.
\end{align*}
For the first term, we write $\rho -2\rho {\phi^\eps}+(\phi^\eps)^2 =(1-\rho){\phi^\eps}^2+ \rho (1-\phi^\eps)^2$. For the second term, we note that
\begin{align*}
  \int_\Omega \Delta {\phi^\eps} \na\phi^\eps\cdot\xi\,dx 
  &  = -\int_\Omega \pa_i {\phi^\eps}  \pa_{ij}\phi^\eps \xi_j\, dx -\int_\Omega \pa_i {\phi^\eps}  \pa_{j}\phi^\eps\pa_i \xi_j\, dx + \int_{\pa\Omega}  \pa_i {\phi^\eps}  \pa_{j}\phi^\eps\xi_j \nu_i \, d\H^{n-1}(x)\\
  &  = -\int_\Omega \na \left(\frac { |\na {\phi^\eps}|^2}{2}\right) \cdot \xi\, dx  -\int_\Omega \na {\phi^\eps}\otimes\na\phi^\eps : D\xi \, dx + \int_{\pa\Omega}  \na {\phi^\eps} \cdot n \na {\phi^\eps}\cdot \xi \, d\H^{n-1}(x)\\
  &  = \frac1 2 \int_\Omega   |\na {\phi^\eps}|^2 \div \xi\, dx  -\int_\Omega \na {\phi^\eps}\otimes\na\phi^\eps : D\xi \, dx + \int_{\pa\Omega}  \na {\phi^\eps} \cdot n \na {\phi^\eps}\cdot  \xi\,  d\H^{n-1}(x)
\end{align*}
 We deduce:
 \begin{align*}
  -\eps^{-1} \int_\Omega (1-2\phi^\eps) \xi \cdot \na \rho
& = \eps^{-1} \int_\Omega [(1-\rho){\phi^\eps}^2+ \rho (1-\phi^\eps)^2]  \div \xi \, dx  + \eps  \int_\Omega   |\na {\phi^\eps}|^2 \div \xi\, dx \\
& \qquad\qquad 
 -2\eps \int_\Omega \na {\phi^\eps}\otimes\na\phi^\eps : D\xi \, dx + 2\eps \int_{\pa\Omega}  \na {\phi^\eps} \cdot n \na {\phi^\eps}\cdot  \xi \, d\H^{n-1}(x).
\end{align*}
When $\alpha=0$, we have  $\na {\phi^\eps} \cdot n=0$ on $\pa\Omega$ so the last term vanishes.  If $\beta=0$, then $\phi=0$ on $\pa \Omega$ and since $\xi$ is tangential to $\pa\Omega$, we have $ \na {\phi^\eps}\cdot  \xi =0$ on $\pa\Omega$ and the last term vanish also.
In both cases, we get formula \eqref{eq:dJN}.
\medskip

In the case of general Robin boundary conditions (with in particular $\beta\neq 0$), we can write the last term as:
\begin{align*}
2\eps \int_{\pa\Omega}  \na {\phi^\eps} \cdot n \na {\phi^\eps}\cdot  \xi \, d\H^{n-1}(x)
& = - 2 \frac{\alpha}{\beta}  \int_{\pa\Omega}   {\phi^\eps}  \na {\phi^\eps}\cdot  \xi \, d\H^{n-1}(x)\\
& = - \frac{\alpha}{\beta}  \int_{\pa\Omega}     \na |\phi^\eps|^2 \cdot  \xi \, d\H^{n-1}(x)
\end{align*}
Formula \eqref{eq:mc} applied to $\Sigma = \pa\Omega$ (since $\xi\cdot n =0$) gives
\begin{align*}
 \int_{\pa\Omega} \div({\phi^\eps}^2 \xi)\,  d\H^{n-1}(x) 
 &  =   \int_{\pa\Omega}  n\otimes n : D (\phi^2 \xi) \,  d\H^{n-1}(x)\\
 &  =   \int_{\pa\Omega} \phi^2  n\otimes n : D \xi \,  d\H^{n-1}(x).
\end{align*}
Writing $ \div({\phi^\eps}^2 \xi) = \na (|\phi^\eps|^2) \cdot \xi + |\phi^\eps|^2 \div \xi$, 
we deduce
\begin{align*}
2\eps \int_{\pa\Omega}  \na {\phi^\eps} \cdot n \na {\phi^\eps}\cdot  \xi \, d\H^{n-1}(x)
& = - \frac{\alpha}{\beta}  \int_{\pa\Omega}   {\phi^\eps}^2  n\otimes n : D \xi \,   d\H^{n-1}(x)
+ \frac{\alpha}{\beta}  \int_{\pa\Omega}     |\phi^\eps|^2 \div \xi \, d\H^{n-1}(x)
\end{align*}
and so
\begin{align*}
 -\eps^{-1} \int_\Omega (1-2\phi^\eps) \xi \cdot \na \rho
& =  \eps^{-1} \int_\Omega [(1-\rho){\phi^\eps}^2+ \rho (1-\phi^\eps)^2]   \div \xi \, dx   + \eps  \int_\Omega   |\na {\phi^\eps}|^2 \div \xi\, dx +\frac{\alpha}{\beta}  \int_{\pa\Omega}     |\phi^\eps|^2 \div \xi \, d\H^{n-1}(x)
 \\
& \qquad\qquad 
 -2\eps \int_\Omega \na {\phi^\eps}\otimes\na\phi^\eps : D\xi \, dx - \frac{\alpha}{\beta}  \int_{\pa\Omega}   {\phi^\eps}^2  n\otimes n : D \xi \,   d\H^{n-1}(x).
\end{align*}
which is \eqref{eq:dJ}.
\end{proof}

\medskip

\section{JKO Scheme and convergence of the discrete time approximation}
The main result of this paper can also be proved at the level of the discrete time approximation constructed in \cite{KMW}.
Such a result can be relevant to some numerical applications so we will state it here.
First, we briefly recall the construction of the JKO scheme:
As usual, $\P(\Omega)$ denotes the set of probability measures 
on $\Omega$ and we denote
\begin{equation}\label{eq:K}
K=\left\{ \rho\in \P(\Omega), \; \rho(x)\leq 1 \mbox{ a.e. in }\Omega\right\}
\end{equation} 
(in particular all $\rho\in K$ are absolutely continuous with respect to the Lebesgue measure and we can identify the measure with its density).
The set $\P(\Omega)$ is equipped with the usual Wasserstein distance, defined by
$$
W_2^2(\rho_1,\rho_2) = \inf_{\pi \in \Pi(\rho_1,\rho_2)}  \int_{\Omega\times\Omega} |x-y|^2 d\pi(x,y)
$$
where $\Pi(\rho_1,\rho_2)$ denotes the set of all probability measures $\pi \in \mathcal P(\Omega\times\Omega)$ with marginals $\rho_1$ and $\rho_2$.

The idea of the JKO scheme is to construct a time-discrete approximation of the solution by successive applications of a minimization problem: For a given initial data $\rho_{in}\in K$,  we fix a time step $\tau>0$ (destined to go to zero) and define the sequence $\rho^n $ by:
\begin{equation}\label{eq:JKO}
\rho^0=\rho_{in} , \qquad  \rho^n \in \mathrm{argmin}\left\{  \frac {1}{2\tau} W_2^2(\rho,\rho^{n-1}) +\F_\eps(\rho)\, ;\, \rho\in K\right\} \qquad \forall n\geq 1.
\end{equation}
%where $W_2(\mu,\nu)$ denotes the usual Wasserstein distance between two probability measures.

The fact that this problem has a minimizer is proved in \cite[Proposition 2.1]{KMW}.
Furthermore,  if $T^n$ is the unique optimal transport map from $\rho^n$ to $\rho^{n-1}$ (that is $T^n\#\rho^n=\rho^{n-1}$ and $W_2^2(\rho^n,\rho^{n-1}) = \int|x-T^n(x)|^2 \rho^n$), we define the velocity
$$ v^n(x) =\frac{x-T^n(x)}{\tau}$$
and the pressure variable $p^n(x)$ such that
%We also define the velocity $v^n$ and pressure $p^n$ such that
%$$
%h \int \rho^n (x)v^n (x) \cdot \na \zeta (x) \, dx =  \int [\rho^n(x)-\rho^{n-1}(x)]\, \zeta (x) \, dx + \mathcal O(\| D^2\zeta\|_\infty W_2^2(\rho^n,\rho^{n-1}))
%$$
%and
$$ 
\rho^n v^n = \eps^{-1}\rho^n \na \phi^n - \na p^n, \qquad p^n \in H^1_{\rho^n}.
$$
with $\phi^n$ solution of \eqref{eq:phi0}. 
The existence of $p^n$ is shown  in \cite[Proposition 2.6]{KMW}.

\medskip

We can then define the piecewise constant function $\rho^{\tau,\eps},p^{\tau,\eps}:[0,T]\mapsto P(\Omega)$ by
\begin{equation}\label{eq:interrhop}
\begin{array}{lll}
\rho^{\tau,\eps}(t)&:=\rho^{n+1} &\text{ for all }  t\in[n\tau,(n+1)\tau)\\
p^{\tau,\eps}(t)& :=p^{n+1} &  \text{ for all }  t\in[n\tau,(n+1)\tau).
\end{array}
\end{equation}

The main result of \cite{KMW} is the convergence of $(\rho^{\tau,\eps},p^{\tau,\eps})$ when $\tau\to0$ with $\eps>0$ fixed to a weak solution of \eqref{eq:weak}-\eqref{eq:phi0}.
The proof of Theorem \ref{thm:conv1} can easily be adapted to establish the convergence of $(\rho^{\tau,\eps},p^{\tau,\eps})$  to a weak solution  of \eqref{eq:HSST} when $\tau$ and $\eps$ both go to zero:
\begin{theorem}[Convergence when $\eps,\tau\to0$]\label{thm:conv3}
Given $T>0$, an initial data $\rho_{in} = \chi_{E_{in}}\in BV(\Omega;\{0,1\})$ and $\mu\geq 0$.
Consider a subsequence $(\eps_n,\tau_n)$ with $\max\{\eps_n,\tau_n\}\to 0$. Up to a subsequence (still denoted $(\eps_n,\tau_n)$),   the discrete time approximation $\rho^{\eps_n,\tau_n}(x,t)$ converges to $\rho(x,t)$ strongly in  $L^\infty((0,T);L^1(\Omega))$ and $q^{\eps_n,\tau_n}$ converges to $q$ weakly-$*$ in $L^2((0,T);(C^s(\Omega))^*)$ (for any $s>0$). 
Furthermore, if the following energy convergence assumption holds:
$$
\lim_{n\to\infty}\int_0^T \J_{\eps_n}( \rho^{\eps_n,\tau_n} (t)) \, dt = \int_0^T \J_0(\rho(t))\, dt,
$$
then $(\rho, q)$ is a weak solution of \eqref{eq:HSST} in the sense of Definition \ref{def:weak2} with initial condition $\chi_{E_{in}}$ and contact angle
$$\gamma=-\min\left( 1,\frac{2\alpha}{\alpha+\sqrt \sigma \beta} \right).$$
\end{theorem}

We will not provide the details of the proof of this result which is a straightforward adaption of the arguments presented in this paper to prove Theorem \ref{thm:conv1}.
The key is to recall that the discrete approximations $\rho^{\eps,\tau}$ and $p^{\eps,\tau}$ satisfy some approximation of equations \eqref{eq:weak11}-\eqref{eq:weak12}. Indeed,
in addition to $\rho^{\eps,\tau}$ and $p^{\eps,\tau}$, we can define  the piecewise constant interpolations 
$\rho^{\tau,\eps}(x,t)$, $p^{\tau,\eps}(x,t)$, $v^{\tau,\eps}(x,t)$ and $\phi^{\tau,\eps}(x,t)$ by
\begin{equation}\label{eq:inter}
\begin{aligned}
%\rho^{\tau,\eps}(t)& :=\rho^{n+1} &\text{ for all } \ \ t\in[n\tau,(n+1)\tau)\\
%p^{\tau,\eps}(t)& :=p^{n+1}  & \text{ for all } \ \ t\in[n\tau,(n+1)\tau)\\
v^{\tau,\eps}(t)& :=v^{n+1} & \text{ for all } \ \ t\in[n\tau,(n+1)\tau)\\
\phi^{\tau,\eps}(t) & :=\phi^{n+1}&  \text{ for all } \ \ t\in[n\tau,(n+1)\tau).
%E ^{\eps,\tau}(t)&:=\rho^{n+1}v ^{n+1}  & \text{ for all }  \ \ t\in[n\tau,(n+1)\tau)\\
\end{aligned}\end{equation}
% \eqref{eq:rhointer}, \eqref{eq:inter}.
and the momentum 
$$E ^{\eps,\tau} (x,t)=\rho^{\tau,\eps}(x,t) v^{\tau,\eps}(x,t).$$
%\begin{equation}
%\begin{aligned}
%\rho^{\eps,\tau}(t)& :=\rho^{n+1}_\eps \qquad \text{if} \ \ t\in[n\tau,(n+1)\tau),\\
%v ^{\eps,\tau}(t,\cdot)&:=\frac{\nabla \psi _\eps^{n+1}}{\tau} \qquad \text{if} \ \ t\in[n\tau,(n+1)\tau),\\
%p ^{\eps,\tau}(t,\cdot )&:=p ^{n+1} _\eps\qquad \text{if} \ \ t\in[n\tau,(n+1)\tau),\\
%E ^{\eps,\tau}(t,\cdot)&:=v ^{n+1}_\eps(t,\cdot)\rho ^{\tau}\qquad\text{if} \ \ t\in[n\tau,(n+1)\tau),
%\end{aligned}
%\end{equation} 
Then we have (see \cite{KMW}):
\begin{proposition}
For any smooth test function $\zeta(x,t) $ compactly supported in $\Omega\times [0,T)$
and given $N$ such that $N\tau \geq T$, there holds:
%, let $\zeta^\tau(x,t)$ be the piecewise constant function defined by $\zeta^\tau(x,t) = \zeta(x,n\tau)$ for $t\in [(n-1)\tau,n\tau)$. Then:
\begin{align}
\int_0^\infty \int_\Omega E^{\eps,\tau}  \cdot \na \zeta \, dx \, dt
&  = -\int_\Omega \rho_{in} (x) \zeta(x,0) \, dx
 -\int_0^\infty  \int_\Omega  \rho^{\eps,\tau}(x,t) \pa_t \zeta(x,t) \, dx\,  dt \nonumber \\
 & \qquad
 + \mathcal O \left(\|D^2\zeta\|_{L^\infty(\Omega \times\R_+)} \sum_{k=0}^N W_2^2(\rho_\eps^n,\rho_\eps^{n-1}) 
+\tau \| \pa_t\zeta\|_\infty + \tau T \| \pa^2_t \zeta\|_\infty
 \right)\label{eq:weakepstau}
 \end{align}
%where $\zeta^\tau(x,t)$ is the piecewise constant function defined by $\zeta^\tau(x,t) = \zeta(x,n\tau)$ for $t\in [(n-1)\tau,n\tau)$ and 
%where $N$ is such that $N\tau =T$ with $\zeta$ supported in $(0,T)$.

For any smooth vector field $\xi(x,t)$ satisfying $\xi\cdot n=0$ on $\pa \Omega$, there holds:
\begin{equation}\label{eq:pressureepstau}\int_0^\infty \int_\Omega E^{\eps,\tau} \cdot \xi\, dx\, dt   = \int _0^\infty \int_\Omega (\eps^{-1}\rho^{\eps,\tau} \na \phi^{\eps,\tau} \cdot \xi  + \mu\rho ^{\eps,\tau}\div \xi\ + p^{\eps,\tau}\div \xi)\, dx\, dt
\end{equation}
\end{proposition}

Passing to the limit in the continuity equation \eqref{eq:weakepstau} can be done exactly as in the case $\tau\to0$ with $\eps>0$ fixed (see \cite{KMW}), while equation \eqref{eq:pressureepstau} is exactly the same as our equation \eqref{eq:weak12}, so we can adapt the proof presented in Section \ref{sec:lim} of the present paper to pass to the limit in 
\eqref{eq:pressureepstau} and prove Theorem~\ref{thm:conv3}.

%We intend to obtain uniform (in $\tau$) estimates and pass to the limit $\tau \to 0$.
%\paragraph{Continuous interpolation.} 
%Interpolating between $\rho^{n}$ and $\rho^{n+1}$ along the natural geodesic for the Wasserstein distance, we

% and the corresponding velocity and momentum variables $(\widetilde{\rho}^{\eps,\tau},\widetilde{v}^{\eps,\tau},\widetilde{E}^{\eps,\tau})$:
%\[
%\widetilde{\rho}^{\tau,\eps}(t)=\left(\frac{t-n\tau}{\tau}(Id-T^{n+1})+T^{n+1}\right){\#}\rho ^{n+1} \qquad \mbox{for} \ \ t\in[n\tau,(n+1)\tau)
%\]
%where we recall that  $T^{n+1}$ is the optimal transport from $\rho^{n+1}$ to $\rho^{n}$. We define $\widetilde{v} ^{\eps,\tau}(t,\cdot)$ as the unique velocity field such that $\widetilde{v} ^{\eps,\tau}(t,\cdot) \in Tan_{\widetilde{\rho} ^{\eps,\tau}}\mathcal{P}_2(\R^d)$ and $(\widetilde{\rho}^{\eps,\tau},\widetilde{v}^{\eps,\tau})$ satisfy the continuity equation, that is:
%\[
%\widetilde{v} ^{\eps,\tau}=v ^{\eps,\tau}\circ \left(\frac{t-n\tau}{\tau}(Id-T^{n+1})+T^{n+1}\right)^{-1}.
%\]
%Finally, we defind the momentum
%\[
%\widetilde{E} ^{\eps,\tau} :=\widetilde{v} ^{\eps,\tau} \,\widetilde{\rho}^{\eps,\tau}.
%\]
%In particular we have
%\[
%\partial_t\widetilde{\rho}^{\eps,\tau}+\nabla \cdot \widetilde{E}^{\eps,\tau}=0
%\]
%in the sense of distribution on $(0,T)\times \Omega$. 

\medskip

\appendix
\medskip
\medskip
\medskip

\section{A few facts about $BV$ functions}
First we recall the classical result:
\begin{proposition}\label{prop:BV}
Let $f_k$ be a sequence of functions such that $f_k \to f$ in $L^1(\Omega)$ when $k\to\infty$.
Then
$$\liminf_{k\to\infty} \int_\Omega |\na f_k| \vphi \, dx 
\geq  \int_\Omega |\na f| \vphi 
$$
for all $\vphi\in C(\Omega)$.
Furthermore, if $\int_\Omega |\na f_k|\, dx \to \int_\Omega |\na f|$, then
$$
\lim_{k\to\infty} \int_\Omega |\na f_k| \vphi \, dx 
= \int_\Omega |\na f| \vphi 
$$
for all $\vphi\in C(\Omega)$.
\end{proposition}
We also need the following result:
\begin{proposition}\label{prop:measure}
Let $f_k$ be a sequence of function such that $f_k \to f$ in $L^1(\Omega)$ and $\int_\Omega |\na f_k|\, dx \to \int_\Omega |\na f|$.
Then
$$
\lim_{k\to \infty}\int_\Omega  \zeta(x) \frac{\pa_i f_k}{|\na f_k|}\frac{\pa_jf_k}{|\na f_k|} |\na f_k| =
 \int_\Omega \zeta(x) \frac{\pa_i f}{|\na f|}\frac{\pa_jf}{|\na f|} |\na f|
 $$
for all $\zeta\in C(\overline \Omega)$.
 \end{proposition}
 It is likely that this proposition is well known, but since we could not find a reference for its proof, we include one below.

\begin{proof}[Proof of Proposition \ref{prop:measure}]
We denote $\nu_i^k= \frac{\pa_if_k}{|\na f_k|}$ which can be understood as the Radon-Nikodym derivative of $\pa_i f_k$ with respect to the measure $|\na f_k|$. Similarly, we note $\nu_i= \frac{\pa_if}{|\na f|}$. 
First, we note that it is enough to prove that under the conditions of the proposition we have
\begin{equation}\label{eq:nfoidl}
\lim_{k\to \infty}
\int_\Omega |\nu_i^k|^2 |\na f_k|  = \int_\Omega |\nu_i |^2 |\na f|\qquad i=1,\dots,n
\end{equation}
since we can write $\nu_i \nu_j =\frac 1 2 [ (\nu_i +\nu_j)^2 -\nu_i^2 -\nu_j^2]$.
%Next, we note that we have
%$$
%\lim_{\eps\to 0} \sum_i
%\int_\Omega |\nu_i^\eps|^2 |\na f_k| =  \lim_{\eps\to 0}  \int_\Omega  |\na f_k|  =   \int_\Omega  |\na f| %= 
%$$
%since $\sum_i  |\nu^\eps_i |^2 = 1$.
% and  $\sum_i  |\nu_i |^2 = 1$ at least $|\na f_k|$ a.e.
%So it is enough to show that
%$$ 
%\liminf_{\eps\to 0}
%\int_\Omega |\nu_i^\eps|^2 |\na f_k|  \geq  \int_\Omega |\nu_i |^2 |\na f|\qquad i=1,\dots,n
%$$

The idea of the proof is as follows:
Given a vector field $ g = (g_1,\dots g_n)\in C^1_0(\Omega,\R^d)$ such that $|g(x)|\leq 1$   for all $x$,
we compute
\begin{align}
- \int_\Omega \pa_i g_i  f_k \, dx = \int_\Omega g_i  \pa_i  f_k \, dx  
& =  \int_\Omega g_i  \nu_i^ k  |\na f_k | \, dx \nonumber \\
& \leq \frac 1 2 \int_\Omega |g_i|^2 |\na f_k | \, dx +\frac 1 2   \int_\Omega |\nu_i^k|^2 |\na f_k | \, dx \label{eq:inecs}
\end{align}
When we sum these inequalities over $i=1,\dots,n$, the quantity on the left converges to $- \int_\Omega  f \div g \, dx $ which is close to $\int_\Omega |\na f|$ for a proper choice of $g$, while the right hand side is bounded by $\int_\Omega |\na f_k|$ (since $|g|^2\leq 1$ and  $|\nu^k|^2 =1$) which converges to $\int_\Omega |\na f|$ by assumption.
Equality in these inequalities will yields the result.
More precisely, if we denote,
$$ 
u^k_i = \frac 1 2 \int_\Omega |g_i|^2 |\na f_k | \, dx +\frac 1 2   \int_\Omega |\nu_i^k|^2 |\na f_k | \, dx + 
\int_\Omega \pa_i g_i  f_k \, dx = \frac 1 2 \int_\Omega |g_i-\nu_i^k|^2 |\na f_k | \, dx \geq 0
$$
then we get
$$ \sum_{i=1}^ d u^k_i \leq \int_\Omega |\na f_k | \, dx + \int_\Omega  f_k \div g \, dx
$$
We now fix $\eta>0$ and choose $g\in C^1_0(\Omega;\R^d)$ such that $|g(x)|\leq 1$ for all $x$ and 
$$  \int_\Omega |\na f| -\eta \leq -\int_\Omega  f \div g \, dx \leq  \int_\Omega |\na f|. $$
Since $f_k\to f $ in $L^1$ and $\int_\Omega  |\na f_k | \, dx \to   \int_\Omega |\na f|$,
there exists $k_0$ (depending on $g$) such that when $k\geq k_0$ then
$$ \sum_{i=1}^ d u^k_i  =\int_\Omega |\na f_k | \, dx + \int_\Omega  f_k \div g \, dx \leq \int_\Omega |\na f | \, dx  +\int_\Omega  f \div g \, dx +\eta \leq 2\eta.$$
In particular
$$  u^k_i  \leq 2\eta \qquad i=1,\dots,d \qquad \forall k\geq k_0.$$
First, this implies that
\begin{align*}
\left| \int_\Omega |g_i|^2 |\na f_k | - \int_\Omega |\nu_i^k|^2  |\na f_k |  \, dx\right|
& \leq 
\int_\Omega \left|  |g_i|^2-|\nu_i^k|^2 \right|  |\na f_k | \, dx \\
& \leq 
\int_\Omega | \left( g_i-\nu_i^k \right) \left( g_i+\nu_i^k \right) ||\na f_k | \, dx\\
& \leq 
\left( \int_\Omega \left( g_i-\nu_i^k \right)^2 |\na f_k | \, dx \right)^{1/2}\left(2 \int_\Omega |\na f_k | \, dx\right)^{1/2}\\
&\leq C\sqrt \eta.
\end{align*}
Using this bound, the inequality \eqref{eq:inecs} and the convergence of $f_k$ to $f$, we deduce  (for $k\geq k_0$):
\begin{align*}
\int_\Omega |\nu_i^k|^2 |\na f_k | \, dx 
&\geq 
\frac 1 2 \int_\Omega |g_i|^2 |\na f_k | \, dx +\frac 1 2   \int_\Omega |\nu_i^k|^2 |\na f_k | \, dx -C\sqrt \eta\\
&\geq  - \int_\Omega \pa_i g_i  f_k \, dx -C\sqrt \eta\\
& \geq  - \int_\Omega \pa_i g_i  f \, dx -C\sqrt \eta
\end{align*}
and 
\begin{align*}
\int_\Omega |\nu_i^k|^2 |\na f_k | \, dx
& \leq 
 \frac 1 2 \int_\Omega |g_i|^2 |\na f_k | \, dx +\frac 1 2   \int_\Omega |\nu_i^k|^2 |\na f_k | \, dx +C\sqrt\eta \\
 &\leq  u_i^k - \int_\Omega \pa_i g_i  f_k \, dx +C\sqrt\eta \\
 &\leq  - \int_\Omega \pa_i g_i  f  \, dx +C\sqrt\eta +C\eta.
\end{align*}

In order to conclude, we note that if we take the constant sequence $f_k=f$ (for all $\eps>0$), then the argument above yields (with the same choice of function $g$, since it only depended on the limit $f$):
$$\int_\Omega |\nu_i|^2 |\na f | \, dx  \geq  - \int_\Omega \pa_i g_i  f \, dx -C\sqrt \eta$$
and 
$$\int_\Omega |\nu_i|^2 |\na f | \, dx \leq  - \int_\Omega \pa_i g_i  f  \, dx +C\sqrt\eta +C\eta$$
We have proved that given $\eta>0$ there exists $k_0$ such that for $k\geq k_0$ we have
$$
 \int_\Omega |\nu_i|^2 |\na f | \, dx - C\sqrt\eta \leq \int_\Omega |\nu_i^k|^2 |\na f_k | \, dx  \leq \int_\Omega |\nu_i|^2 |\na f | \, dx + C\sqrt\eta
$$
which implies \eqref{eq:nfoidl} and complete the proof.
\end{proof}

\medskip

\section{A Lions-Aubin compactness result}
The following result is a simple adaptation of some standard result. We provide a proof for the sake of completeness.
\begin{lemma}\label{lem:LA}
Let $u_n$ be a sequence of function bounded in $ L^\infty(0,T ;L^1( \Omega))$ such that
$ u_n$ is bounded in $L^\infty((0,T);\BV(\Omega))$ and $u_n\to u$ in $L^\infty((0,T);H^{-1}(\Omega))$
Then
$$ \sup_{t\in[0,T]} \|u_n(t)-u(t)\|_{L^1 (\Omega)} \to 0$$
%, converging weakly to $u$ and such that
%$u_n=w_n+v_n$ with
%$$\| w_n \|_{L^1((0,T);BV( \Omega))} \leq C, \quad \| v_n\|_{L^2((0,T)\times \Omega)} \to 0, \quad \| u_n\|_{C^{1/2}(0,T;H^{-1}(\Omega))}\leq C.$$ 
%Then
%$$ u_n \to u \quad\mbox{ strongly in } L^1((0,T)\times \Omega)$$
\end{lemma}
\begin{proof}
Since $\int |\na u_n(t)| \leq C$ and $u_n(t)$ converges to $u(t)$ in $H^{-1}(\Omega)$, we can show that
$u(t)\in BV(\Omega)$ and $\int |\na u(t)| \leq C$.

Next we note that for any $\delta>0$, there exists $C_\delta$ such that
$$
\| v \| _{L^1(\Omega) } \leq\delta  \| v \|_{\BV(\Omega) } + C_\delta  \| v \|_{H^{-1}(\Omega)}
 $$
for all $v\in\BV(\Omega)$.

In particular, this gives:
$$
\| u_n(t) - u(t) \| _{L^1(\Omega) } \leq\delta  \| u_n(t) - u(t) \|_{\BV(\Omega) } + C_\delta  \| u_n(t) - u(t) \|_{H^{-1}(\Omega)}
 $$
and the $BV$ bound, together with the strong convergence in $H^{-1}(\Omega)$ implies
$$
\limsup  
\| u_n(t) - u(t) \| _{L^\infty((0,T);L^1(\Omega)) } 
 \leq C\delta.
$$
Since this holds for all $\delta>0$, we deduce
$$\limsup  
\| u_n - u \| _{L^\infty((0,T);L^1(\Omega))} 
=0
$$
and the result follows.

\end{proof}

\medskip

\section{$\Gamma$-convergence of $\J_\eps$}
\label{app:energy}
We wish to prove the following proposition which gives the $\Gamma$ convergence of $\J_\eps$ to $\J_0$:
\begin{proposition}\label{prop:Gamma}
The following holds:
\item[(i)] For any family $\{ \rho^\eps\}_{\eps>0}$ that converges to $\rho$ in $L^1(\Omega)$,
 $$\liminf_{\eps\to 0} \J_\eps(\rho^\eps) \geq  \J_0 (\rho).$$

\item[(ii)] Given   $\rho\in L^1(\Omega)$, there exists a sequence $\{ \rho^\eps\}_{\eps>0}$ that converges to $\rho$ in $L^1(\Omega)$ such that
 $$\limsup_{\eps\to 0} \J_\eps(\rho^\eps) \leq  \J_0 (\rho).$$
\end{proposition}
We recall that this proposition is proved in \cite{MW} (Proposition 5.3) when $\J_\eps$ is restricted to characteristic functions. We show below how the proof can be adapted to our more general case.

\begin{proof}[Proof of Proposition \ref{prop:Gamma}]
First, we note that the limsup properties (part (ii)) follows from the corresponding result in \cite[Proposition 5.3]{MW} .
Indeed, if $\rho\notin \BV(\Omega;\{0,1\})$, then $\J_0(\rho)=\infty$ and there is nothing to prove,
while if $\rho\in \BV(\Omega;\{0,1)\}$ then $\rho=\chi_E$ for some $E$ satisfying $P(E)<\infty$, so 
\cite[Proposition 5.3]{MW}  applies.
%formula \eqref{eq:Jepsd} shows that $\limsup_{\eps\to 0} \J_\eps(\rho)=\infty$ so that we can take $\rho^\eps=\rho$. 

\medskip

To prove the liminf property (part (i)), we need to slightly modify the proof of 
\cite[Proposition 5.3]{MW} by using the formula
\eqref{eq:Jepsd2} instead of \eqref{eq:Jepsd} for $\J_\eps$.
We only provide details in the case of Neumann and Dirichlet conditions (the Robin boundary condition is then proved combining both arguments, as in \cite{MW}).
\medskip

\noindent{\bf Neumann boundary conditions ($\alpha=0$).} 
If $\liminf_{\eps\to 0} \J_\eps(\rho^\eps)=\infty$, there is nothing to prove, so we can assume (up to a subsequence) that $\J_\eps(\rho^\eps)\leq C$ and $\liminf_{\eps\to 0} \J_\eps(\rho^\eps)<\infty$.
Up to another subsequence, we can also assume that $\rho^\eps \to \rho$ a.e. in $\Omega$.

Next, we recall that  \eqref{eq:Jepsd2} gives 
\begin{equation}\label{eq:Jneu}
\J_\eps(\rho)  = \frac{1}{2\sigma\eps}\int_\Omega  (1-\rho) (\sigma\phi)^2 + \rho(1-\sigma \phi)^2 \, dx +\frac{ \eps}{2} \int_\Omega |\na \phi| ^2\, dx
\end{equation}
and introducing the functions (both defined for $t\in [0,1]$)
$$f(t)=2 \min(t,1-t), \qquad F(t) = \int_0^t f(\tau)\, d\tau = \begin{cases} t^2 & \mbox{ for } 0\leq t\leq 1/2 \\ 2t-t^2 -\frac 1 2 & \mbox{ for } 1/2 \leq t\leq 1\end{cases} $$
we find  (see \eqref{eq:BVphi}):
\begin{equation}\label{eq:lNJK}
\frac{1}{2\sigma^{3/2}} \int_\Omega |\na F(\sigma {\phi^\eps})| \, dx\leq  \J_\eps (\rho^\eps) .
\end{equation}
%It follows that the sequence $F(\sigma\phi^\eps)$ is bounded in $\BV(\Omega)$.

\medskip

Furthermore, \eqref{eq:Jepsd} give
$$  \J_\eps (\rho)   =  \frac{1}{2\sigma \eps}\int_\Omega  \rho(1-\rho)\, dx
+  \frac{1}{2\sigma \eps}  \int _{\Omega}(\rho-\sigma \phi)^2 \, dx  + 
\frac{ \eps}{2}  \int_\Omega |\na \phi| ^2\, dx
$$
which implies
$$
 \int _{\Omega}(\rho^\eps-\sigma \phi^\eps)^2 \, dx\leq  2\sigma \eps \J_\eps(\rho^\eps)\leq C\eps$$
and 
$$\int_\Omega  \rho^\eps(1-\rho^\eps)\, dx\leq 2\sigma \eps \J_\eps(\rho^\eps)\leq C\eps.$$
The first inequality implies that $\sigma \phi^\eps$ converges in $L^2$ to $\rho$. The second inequality implies that
that $\rho=0$ or $1$ a.e. in $\Omega$.\

We deduce that $F(\sigma\phi^\eps)$ converges (strongly in $L^1$ for example) to $F(\rho)=\frac 1 2 \rho$ (since $F(0)=0$ and $F(1) =1/2$), and \eqref{eq:lNJK} gives
$$
\liminf_{\eps\to 0} \J_\eps (\rho^\eps) \geq\frac{1}{2\sigma^{3/2}}  \int_\Omega |\na F(\rho)| \, dx = 
\frac{1}{4\sigma^{3/2}} \int_\Omega |\na \rho| \, dx = \J_0(\rho) .
$$

\medskip

\noindent{\bf Dirichlet boundary conditions.} 
We still have \eqref{eq:Jneu} and thus \eqref{eq:lNJK} in this case, but since $\phi^\eps=0$ on $\pa\Omega$, we can extend the function $\phi^\eps$ by zero outside $\Omega$. Denoting by $\overline{\phi^\eps}$ this extension, we find 
$$\J_\eps (\rho^\eps)\geq \frac{1}{2\sigma^{3/2}} \int_\Omega |\na F(\phi^\eps)| \, dx = \frac{1}{2\sigma^{3/2}} \int_{\R^n} |\na F(\overline{\phi^\eps})| \, dx $$
and so (proceeding as above)
\begin{align*}
\liminf_{\eps\to 0} \J_\eps (\rho^\eps) \geq \frac{1}{2\sigma^{3/2}}  \int_{\R^n} |\na F(\rho)| \, dx & =  \frac{1}{4\sigma^{3/2}}  \int_{\R^n} |\na \rho| \, dx  \\
& = \frac{1}{4\sigma^{3/2}}  \left[ \int_{\Omega} |\na \rho| \, dx + \int_{\pa\Omega} \rho\, dH^{n-1}(x)\right]\\
& = \J_0(\rho).
\end{align*}

\end{proof}

%Since $v_n\to 0$ strongly in $L^1((0,T)\times \Omega)$, we need to show that $w_n \to u $  strongly in $L^1((0,T)\times \Omega)$.
%Clearly, we have that $w_n$ converges to $u$ weakly in $L^1$ and since $ \| w_n \|_{L^1((0,T);BV( \Omega))}$, we get that $u\in L^1((0,T);BV( \Omega))$. In particular, we have
%$$\| w_n -u \|_{L^1((0,T);BV( \Omega))} \leq C.$$

%Next, we note that for any $\delta$, there exists $C_\delta$ such that
%$$
%\| w_n(t) - u(t) \| _{L^1(\Omega) } \leq\delta  \| w_n(t) - u(t) \|_{BV(\Omega) } + C_\delta  \| w_n(t) - u(t) \|_{H^{-1}(\Omega}
% $$
%and so
%\begin{align*}
%\| w_n - u \| _{L^1((0,T)\times \Omega) } 
%& \leq \delta  \| w_n - u\|_{L^1(0,T;BV(\Omega)) }+ 
%C_\delta\int_0^T  \| u_n(t) - u(t) \|_{H^{-1}}+ \| v_n(t) \|_{H^{-1}}\, dt\\
%& \leq C \delta + 
%C_\delta\int_0^T  \| u_n(t) - u(t) \|_{H^{-1}}+ \| v_n(t) \|_{L^2(\Omega)}\, dt
%\end{align*}
%Finally, the bound $\| u_n\|_{C^{1/2}(0,T;H^{-1}(\Omega))}\leq C$ implies that $u_n(t)$ converges to $u(t)$ uniformly in $H^{-1}(\Omega)$. Together with the strong convergence of $v_n$, this implies
%$$
%\limsup  
%\| w_n - u \| _{L^1((0,T)\times \Omega) } 
% \leq C\delta
%$$
%and since this holds for all $\delta>0$, we deduce
%$$\limsup  
%\| w_n - u \| _{L^1((0,T)\times \Omega) } 
%=0
%$$
%and the result follows.  

\bibliographystyle{plain}
\bibliography{JKO_chemotaxis}

\end{document}